\documentclass{IEEEtran}
\usepackage{amsmath,amssymb,amsthm,subfigure}

\usepackage{tikz}
\usepackage{xypic}
\usepackage{cases}
\usepackage{epstopdf,graphicx}

\newtheorem{theorem}{Theorem}[section]
\newtheorem{corollary}[theorem]{Corollary}

\newtheorem{lemma}[theorem]{Lemma}

\newtheorem{assumption}[theorem]{Assumption}

\newcommand{\bracket}[1]{\ensuremath{\left[ #1 \right]}}
\newcommand{\braces}[1]{\ensuremath{\left\{ #1 \right\}}}
\newcommand{\parenth}[1]{\ensuremath{\left( #1 \right)}}

\newcommand{\SO}{\operatorname{SO}(3)}
\newcommand{\MI}{{\mathbb I}}
\newcommand{\refeqn}[1]{(\ref{eqn:#1})}
\renewcommand{\Re}{\mathbb R}
\newcommand{\tr}[1]{\mathrm{tr}\ensuremath{\!\bracket{#1}}}

\newcommand{\EditTL}[1]{{\color{blue}\protect #1}}
\renewcommand{\EditTL}[1]{{\protect #1}}

\graphicspath{{./Figs/}}

\title{Semi-Global Attitude Controls Bypassing the Topological Obstruction on  $\SO$}

\author{Taeyoung Lee, Dong Eui Chang and Yongsoon Eun

\thanks{T. Lee is with the Department of Mechanical and Aerospace Engineering,  George Washington University, Washington DC, USA.
        {\tt\small  tylee@gwu.edu}}
\thanks{D.E. Chang is with the School of Electrical Engineering, KAIST, Daejeon, Korea.
        {\tt\small dechang@kaist.ac.kr}}
 \thanks{Y. Eun is with the Department of Information \& Communication Engineering, DGIST, Daegu, Korea.
        {\tt\small yeun@dgist.ac.kr}}
}

\begin{document}

\maketitle

\begin{abstract}
This paper presents global tracking strategies for the attitude dynamics of a rigid body. It is well known that global attractivity is prohibited for continuous attitude control systems on the special orthogonal group. Such topological restriction has been dealt with either by constructing smooth attitude control systems that exclude a set of  measure zero in the region of attraction, or by introducing hybrid control systems to obtain global asymptotic stability. This paper proposes alternative attitude control systems that are continuous in time to achieve exponential stability, where the region of attraction covers the entire special orthogonal group. The main contribution of this paper is providing a new framework to overcome the topological restriction in attitude controls without relying on discontinuities through the controlled maneuvers. The efficacy of the proposed methods is illustrated by numerical simulations and experiments.

\end{abstract}

\section{Introduction}

The   attitude dynamics  and control of a rigid body encounters  the unique challenge that the configuration space of  attitudes cannot be globally identified with a Euclidean space. They  evolve on the compact nonlinear manifold, referred to as the three-dimensional special orthogonal group or $\SO$, that is composed of $3\times 3$ orthogonal matrices with the determinant of one.

Traditionally, the special orthogonal group has been parameterized via local coordinates, such as Euler angles or Rodriguez parameters. It is well known that such minimal, three-parameter representations suffer from singularities~\cite{StuSR64}.

Quaternions have been regarded as an ideal alternative to minimal attitude representations, as they are defined by four parameters and they do not exhibit singularities. However, the configuration space of quaternions, namely the three-sphere double covers the special orthogonal group, and consequently, there are two antipodal quaternions corresponding to the same attitude. In fact, it has been shown that 5 parameters are required at least to represent the special orthogonal group globally in a one-to-one manner~\cite{StuSR64}.  This ambiguity inherent to quaternions should be carefully resolved. Otherwise, there could occur unwinding phenomena, where the rigid body rotates unnecessarily through a large angle even if the initial attitude error is small~\cite{BhaBerSCL00}. This has been handled by constructing a control input such that the two antipodal quaternions yield the same input~\cite{WenKreITAC91}, which is equivalent to designing an attitude controller in terms of rotation matrices. Another approach is to define an exogenous mechanism to lift attitude measurements on the special orthogonal group into the three-sphere in a robust fashion~\cite{MaySanITAC11}, which may cause additional complexities.

Alternatively, attitude control systems have been developed directly on the special orthogonal group to avoid the singularities of minimal parameterizations and the ambiguity of quaternions concurrently~\cite{BulLew05}. More specifically, a configuration error function that measures the discrepancy between the desired attitude and the current attitude is formulated via a matrix norm, and control systems  are designed such that controlled trajectories are attracted to the minimum of the error function, thereby accomplishing asymptotic stability.

However, regardless of the choice of attitude representations, attitude control systems are constrained by the topological restriction on the special orthogonal group that prohibits achieving global attractivity via any continuous feedback control~\cite{WilJDE67,BhaBerSCL00,BerPerA13}. This is because the domain of attraction for those systems is homeomorphic to a Euclidean space, which cannot be identified with the tangent bundle of the special orthogonal group globally. For example, in the design of attitude control systems based on the aforementioned configuration error function, there are at least four critical points of the error function, according to the Lusternik-Schnirelmann category~\cite{TakIM68}. This implies that there will be at least three undesired equilibria in the controlled dynamics, and the region of attraction to the desired attitude excludes the union of the stable manifolds of those undesired equilibria. One can show that the resulting reduced region of attraction almost covers the special orthogonal group~\cite{WenKreITAC91,Kod88,BulMurA99}, while precluding only a set of zero measure. But, the existence of the stable manifolds of the undesired equilibria may affect the controlled dynamics strongly~\cite{LeeLeoPICDC11}.

Recently, the topological restriction has been tackled by introducing discontinuities in the controlled attitude dynamics. In particular, a set of attitude configuration error functions, referred to as synergistic potential functions, has been proposed~\cite{MayTeeITAC13}. The family of synergistic potential functions is constructed such that at each undesired critical point of a potential function, there is another member in the family with a lower value of the error. By consistently switching to the controller derived from the minimal potential function, robust global asymptotic stability is achieved for the attitude dynamics. While this approach avoids chattering behaviors by introducing hysteresis, discontinuities in the control input may excite unmodeled dynamics and cause undesired behaviors in practice. Interestingly, it has been unclear if a more general class of feedback control systems could accomplish the global stabilization task without introducing such disruptions in the control input~\cite{MayTeeA13}.

The objective of this paper is to present an alternative framework to overcome the topological restriction on the special orthogonal group with control inputs that are continuous in time. This is achieved by shifting the desired attitude temporarily, instead of modifying the attitude error functions as in~\cite{MayTeeITAC13,LeeITAC15}. More explicitly, when the initial attitude and the initial angular velocity do not belong to the estimated region of attraction of a smooth attitude controller, the desired trajectory is shifted to a trajectory that is sufficiently close to the initial condition to guarantee convergence. While the initial value of the shifted reference trajectory is distinct from that of the true reference trajectory, it is constructed as a time-varying function such that the shifted reference trajectory exponentially converges to the true reference trajectory as  time tends to infinity. Consequently, the corresponding continuous-time controlled trajectory, which is designed to follow the shifted reference trajectory, will converge to the true reference trajectory. The resulting time-varying attitude control system is discontinuous with respect to the initial condition, thereby bypassing the topological restriction. But, it is continuous in time so as to avoid the aforementioned issues of switching in hybrid attitude controls.

All of these are rigorously examined and analyzed  so as to show exponential convergence to the true reference trajectory and to verify that the region of attraction covers the special orthogonal group completely. Furthermore, the shifted reference trajectory is formulated in  such an explicit manner, using a conjugacy class in the special orthogonal group, that no complicated inequality conditions are needed.  Later, this approach is also extended to adaptive controls to handle  unknown constant disturbances in the attitude dynamics. In short, the unique contribution of the proposed approach is that global attractivity is accomplished on the special orthogonal group with control inputs that are continuous in time, overcoming the topological restriction.

This paper is organized as follows. Mathematical preliminaries are presented and the attitude control problem is formulated in Section \ref{sec:PF}. With the assumption that there is no disturbance, two types of attitude control strategies are proposed in Section \ref{sec:ATC}. These are extended to adaptive controls in Section \ref{sec:AATC}, followed by numerical examples and experimental results.

\section{Problem Formulation}\label{sec:PF}

\subsection{Mathematical Preliminaries}

 The inner product $\langle A, B\rangle $ of two matrices or vectors $A$ and $B$ of the same size denotes the usual Euclidean inner product, i.e., $\langle A, B\rangle = \operatorname{tr}(A^TB)$. The norm  $\| A\|$ for a matrix or vector $A$ denotes the Euclidean norm, i.e., $\|A\|^2 = \langle A, A\rangle = \operatorname{tr}(A^TA)$. The minimum eigenvalue of a symmetric matrix $A$ is denoted by $\lambda_{\rm min} (A)$ and the maximum eigenvalue by $\lambda_{\rm max} (A)$.

The attitude dynamics of a rigid body evolve on the three-dimensional special orthogonal group, $\SO=\{R\in\Re^{3\times 3}\,|\, R^T R=I_{3\times 3},\, \mathrm{det}[R]=1\}$. For any $R, R_1, R_2 \in \SO$,
\begin{equation}\label{R:isometry}
\|RR_1 - RR_2\| = \| R_1  - R_2 \| = \|R_1R - R_2R\|.
\end{equation}
 Let  $\mathfrak{so}(3)$  denote the set of all $3\times 3$ skew symmetric matrices. The hat map $\hat {}{}: \mathbb R^3 \rightarrow \mathfrak{so}(3)$ is defined by
\[
v = (v_1, v_2,v_3) \mapsto \hat v = \begin{bmatrix}
0 & -v_3 & v_2 \\
v_3 & 0 & -v_1 \\
-v_2 & v_1 & 0
\end{bmatrix},
\]
and its inverse map is denoted by $\vee$ and called the vee map.   For any $v$ and $w$ in $\mathbb R^3$, $\hat v w = v \times w$ and
\[
\langle \hat v, \hat w\rangle  = 2\langle v, w\rangle,
\]
where  the left side is the inner product  on $\mathbb R^{3\times 3}$ and the  right   on $\mathbb R^3$. For $\theta \in [0,2\pi]$ and a unit vector $v\in \mathbb R^3$,  the matrix exponential  $\exp(\theta \hat v)$ is computed as follows:
\[
\exp(\theta \hat v) =
I +\sin\theta\hat v + (1-\cos\theta)\hat v^2.
\]

Next, we recall conjugacy classes in $\SO$~\cite{Cur84}. Let $Z_\theta\in\SO$ be the rotation about the axis $e_3=(0,0,1)$ by an angle $\theta\in\Re$:
\[
Z_\theta = \exp(\theta\hat e_3) =
\begin{bmatrix}
\cos\theta & -\sin\theta & 0 \\
\sin\theta & \cos\theta & 0\\
0&0&1
\end{bmatrix}.
\]
It is straightforward to show
\begin{equation}\label{Z:distance}
\|Z_{\theta_1} - Z_{\theta_2}\| =2 \sqrt{1-\cos(\theta_1 - \theta_2)}
\end{equation}
for any $\theta_1, \theta_2\in \mathbb R$.
For $\theta\in \mathbb R$, define the  conjugacy class of  $Z_\theta$ in $\SO$ as
\[
C_\theta = \{ R \in \SO \mid R = UZ_\theta U^T,\, U\in\SO\},
\]
which is the set of all rotations through angle $\theta$. The group $\SO$ is partitioned into conjugacy classes. More explicitly,
\begin{equation}\label{C:union}
\SO = \bigcup_{\theta \in [0,\pi]}C_\theta
\end{equation}
and
\[
C_{\theta_1} \bigcap C_{\theta_2} = \emptyset
\]
for any $0\leq \theta_1 < \theta_2 \leq \pi$.

Therefore, for any $X\in\SO$, there exist a unique angle  $\theta \in [0,\pi]$ and some $U\in\SO$ such that
\begin{equation}\label{XUZU}
X = U Z_{\theta} U^T,
\end{equation}
where $\theta \in [0,\pi]$ is determined by
\[
\theta = \arccos \left ( \frac{\operatorname{tr}(X)-1}{2}\right ).
\]
The rotation matrix $U$ satisfying (\ref{XUZU}) is not unique, but one can be obtained as follows. Let $v$ be a unit eigenvector of $X$ corresponding to eigenvalue 1. This vector $v$ satisfies $\exp(\theta \hat v) = X$ or $X^T$. If $\exp(\theta \hat v) =  X$ then set $u_3 = v$. Otherwise, set $u_3 = -v$. Choose a unit vector $u_1$ perpendicular to $u_3$ and let $u_2 = u_3 \times u_1$. Then, the rotation matrix $U$  defined by
\[
U = \begin{bmatrix}
u_1 & u_2 & u_3
\end{bmatrix}
\]
satisfies (\ref{XUZU}).
Alternatively, for $0< \theta <\pi$ the vector $u_3$ can be  computed as follows:
\[
 u_3 =  \left ( \frac{X-X^T}{2\sin\theta}  \right )^\vee,
 \]
and the remaining columns $u_1$ and $u_2$ are constructed as discussed above.

It is easy to show that
\[
\max_{R_1, R_2 \in \SO} \|R_1 - R_2\| = 2\sqrt{2}
\]
and the maximum value $2\sqrt 2$ is attained if and only if
\[
R_1R_2^T \in C_\pi,
\]
 where $C_\pi$ is the conjugacy class of $Z_\pi$.

\subsection{Attitude Dynamics and Control Objective}

The equations of motion for the attitude dynamics of a rigid body are given by
\begin{align}
\dot R  & = R\hat \Omega,\label{eqn:R_dot}\\
\MI \dot \Omega &= (\MI \Omega) \times \Omega + \tau+\Delta,\label{eqn:Omega:dot}
\end{align}
with the rotation matrix  $R \in \SO$ representing the linear transformation of the representation of a vector from the body-fixed frame to the inertial frame,  and the angular velocity $\Omega \in \mathbb R^3$ of the rigid body resolved in the body-fixed frame. The moment of inertia matrix is denoted by $\MI \in\mathbb R^{3\times 3}$, which is symmetric and positive-definite, and the control torque resolved in the body-fixed frame is denoted by $\tau \in \mathbb R^3$.

The above equations include a constant but unknown disturbance torque $\Delta\in\Re^3$, which satisfies the following assumption.
\begin{assumption}\label{assump:delta}
The magnitude of the disturbance is bounded by a known constant $\delta>0$, i.e. $\|\Delta\|\leq \delta$.
\end{assumption}

Let $(R_d(t), \Omega_d(t)) \in \SO \times \mathbb R^3$ be a smooth reference trajectory such that
\begin{equation}
\dot R_d(t) = R _d(t) \hat \Omega_d (t),\label{eqn:Rd_dot}
\end{equation}
for all $t\geq 0$. We wish to design a control torque $\tau$ such that the reference trajectory becomes asymptotically stable.

\section{Attitude Tracking Controls}\label{sec:ATC}

Throughout this section, it is assumed that there is no disturbance in the dynamics, i.e., $\Delta=0$. A smooth attitude controller that yields almost global exponential stability is first presented, and it is extended for global attractivity.

\subsection{Almost Global Tracking Strategy}\label{sec:AGT}

The attitude tracking error $E_R\in\Re^{3\times 3}$ and the angular velocity tracking error $e_\Omega\in\Re^3$ are defined as
\[
E_R = R - R_d, \qquad e_\Omega = \Omega - \Omega_d.
\]
Notice that our definition of $e_\Omega$ is distinct from that in \cite{LeeLeoPICDC10}, where the desired angular velocity is multiplied by $R^TR_d$.  Define an auxiliary vector $e_R\in\Re^3$ as
\begin{equation}
e_R  = \frac{1}{2} (R^T_dR - R^TR_d)^\vee.\label{eqn:eR}
\end{equation}
The relations between $E_R$ and $e_R$ are summarized as follows.
\begin{lemma}\label{lemma:eR:ER}
1. For any $R$ and $R_d \in \SO$,
\[
\|e_R\|^2 = \frac{1}{2}\|R - R_d\|^2 \left ( 1 - \frac{1}{8}\|R -R_d\|^2\right ).
\]

2.  For any $R$ and  $R_d \in \SO$,
\[
\|e_R\|^2  \leq \frac{1}{2}\|R -R_d\|^2.
\]

3. For any number $a$ satisfying $0 < a < 1$ and  for any $R$ and  $R_d \in \SO$ satisfying $ \|R -R_d\| \leq 2\sqrt{2a}$,
\[
\frac{(1-a)}{2}\|R -R_d\|^2 \leq  \|e_R\|^2.
\]

\begin{proof}
Let $R^TR_d=\exp(\theta\hat v)$ for $\theta\in[0,\pi]$  and $v\in\Re^3$ with $\|v\|=1$. Using Rodrigues' formula, one can show 
\[
\|E_R\|^2=4(1-\cos\theta),\quad \|e_R\|= \sin\theta.
\]
Substituting these, it is straightforward to show the first identity, which implies the next two inequalities. 
\end{proof}
\end{lemma}

\begin{lemma}
 Along the trajectory of the  rigid body  system,
\[
{\dot e}_R = C(R^TR_d) e_\Omega + e_R \times \Omega_d
\]
where
\begin{equation}\label{def:C:matrix}
C(R^TR_d) = \frac{1}{2}(\operatorname{tr} (R^TR_d) I - R^TR_d).
\end{equation}

\begin{proof}
From \refeqn{R_dot} and \refeqn{Rd_dot},
\[
\dot e_R =
\frac{1}{2}\{\hat\Omega R^T R_d + R_d^T R \hat\Omega
-\hat\Omega_d R_d^T R - R^T R_d\hat\Omega_d\}.
\]
Substitute $\Omega= e_\Omega +\Omega_d$, and then the two terms dependent on $e_\Omega$ reduce to $C(R^TR_d) e_\Omega$ by the identity
\[
(\hat x A + A^T \hat x)^\vee = (\operatorname{tr}(A)I - A)x
\]
for all $x\in \mathbb R^3$ and $A\in \mathbb R^{3\times 3}$.  The remaining terms, which depend on $\Omega_d$, simplify to $e_R \times \Omega_d$ by  the definition of $e_R$ in \refeqn{eR} and the identity, $\hat x\hat y - \hat y \hat x = \widehat{x\times y}$ for any $x,y\in\Re^3$.
\end{proof}
\end{lemma}

Consider a Lyapunov  function (candidate) $V: \SO \times \mathbb R^3 \times \mathbb R \rightarrow \mathbb R$ defined by
\begin{align*}
V(R, \Omega,t) &= \frac{k_R}{4}\| E_R\|^2 + \frac{1}{2} \|e_\Omega\|^2  + \mu \langle e_R,   e_\Omega \rangle,\end{align*}
where $k_R >0$ and $\mu>0$. Define an auxiliary function $V_0(R, \Omega,t) $ as follows:
\begin{equation}\label{def:V0}
V_0 (R, \Omega,t) = \frac{k_R}{4}\| E_R\|^2 + \frac{1}{2}\| e_\Omega\|^2,
\end{equation}
which coincides with $V$ when $\mu = 0$.  The following lemma discusses positive-definiteness of the function $V$ and its relationship with its auxiliary $V_0$.

\begin{lemma}\label{lemma:V:positive}
Suppose
\[
0 < \mu < \sqrt {k_R}.
\]
Then, the symmetric matrices
\begin{equation}\label{def:W1}
W_1 = \begin{bmatrix}
\frac{1}{4}k_R & -\frac{1}{2\sqrt 2}\mu \\
-\frac{1}{2\sqrt 2}\mu & \frac{1}{2}
\end{bmatrix},\,\,
W_2 = \begin{bmatrix}
\frac{1}{4}k_R & \frac{1}{2\sqrt 2}\mu \\
\frac{1}{2\sqrt 2}\mu & \frac{1}{2}
\end{bmatrix}
\end{equation}
are positive-definite and satisfy
\begin{equation}\label{W1:less:than:V}
z^T W_1z \leq V(R,\Omega, t)\leq z^T W_2 z
\end{equation}
for all $(R,\Omega) \in \SO\times \mathbb R^3$ and $t\geq 0$, where $z = ( \| E_R  \|, \| e_\Omega \| ) \in \mathbb R^2$. Moreover,
\begin{equation}\label{V0:less:V}
\frac{\sqrt{k_R} - \mu}{\sqrt{k_R}}V_0(R,\Omega,t) \leq V(R,\Omega, t)\leq \frac{\sqrt{k_R} + \mu}{\sqrt{k_R}}V_0(R,\Omega,t)
\end{equation}
for all $(R,\Omega) \in \SO\times \mathbb R^3$ and $t\geq 0$.

\begin{proof}
The inequality (\ref{W1:less:than:V}) follows from  Lemma \ref{lemma:eR:ER} and the Cauchy-Schwarz inequality. Next, let $\tilde z = (\frac{\sqrt{k_R}}{2}\|E_R\|, \frac{1}{\sqrt{2}} \|e_\Omega\|)^T\in\Re^2$. Then, $z^T W_1 z$ can be rewritten as
\[
z^T W_1 z = (\tilde z)^T \begin{bmatrix} 1 & -\frac{\mu}{\sqrt{k_R}} \\
 -\frac{\mu}{\sqrt{k_R}} & 1 \end{bmatrix} \tilde z \geq \left(1-\frac{\mu}{\sqrt{k_R}}\right) \|\tilde z\|^2,
\]
which, together with (\ref{W1:less:than:V}), shows the first inequality in  (\ref{V0:less:V})  as $\|\tilde z\|^2=V_0(R,\Omega,t)$. The second inequality in  (\ref{V0:less:V}) can be shown similarly.
\end{proof}
\end{lemma}

We propose the following tracking controller:
\begin{equation}\label{tracking:control}
\tau =  - (\MI \Omega)\times \Omega + \MI ( -k_R e_R - k_\Omega e_\Omega + \Omega \times \Omega_d + \dot \Omega_d),
\end{equation}
where $k_R>0$ is the same constant as that used for the function $V$,  and $k_\Omega >0$. The following lemma computes the rate of change of $V_0$ and $V$ along the trajectory of the closed-loop system.

\begin{lemma}\label{lemma:Vdot:less}
1. Along the trajectory of the closed-loop system with the control (\ref{tracking:control}), the time-derivative of the auxiliary $V_0$ is given by
\begin{equation}
{\dot V}_0  (R,\Omega,t) =-k_\Omega\|e_\Omega\|^2.\label{eqn:V0_dot}
\end{equation}
2. Let $a$ be any number satisfying $0 < a <1$ and choose any $\mu$ such that
\[
0 < \mu < \frac{4(1-a)k_R k_\Omega}{4(1-a) k_R + k_\Omega^2}.
\]
If
\begin{equation}
\| R -R_d\| \leq 2\sqrt{2a},  \label{eqn:ER2a}
\end{equation}
 then the time-derivative of the Lyapunov function along the controlled trajectories satisfies
\begin{equation}\label{V:dot:lessthan:W3}
{\dot V}  (R,\Omega,t) \leq -z^TW_3z
\end{equation}
with  $z = (\| E_R \|, \| e_\Omega \|) \in \mathbb R^2$ and the matrix $W_3\in\Re^{3\times 3}$ defined as
\begin{equation}\label{def:W3}
W_3 = \begin{bmatrix}
 \frac{(1-a)}{2}\mu k_R & - \frac{1}{2\sqrt 2}\mu k_\Omega \\
- \frac{1}{2\sqrt 2}\mu k_\Omega & k_\Omega -\mu
\end{bmatrix},
\end{equation}
which is positive-definite.

\begin{proof}
Let $Q=R_d^T R\in\SO$. From the attitude kinematics equations and the definition of $e_\Omega$,
\[
\dot Q = \dot R_d^T R + R_d^T\dot R = -\hat\Omega_d Q + Q\hat\Omega= Q\hat e_\Omega + Q\hat \Omega_d-\hat\Omega_d Q.
\]
We have 
\[
\|E_R\|^2 = \tr{(R-R_d)^T(R-R_d)} = 2\tr{I_{3\times 3}-Q}.
\]
 Therefore,
\begin{align*}
\frac{d}{dt} \parenth{\frac{k_R}{4}\|E_R\|^2} 
&= -\frac{k_R}{2}\tr{\dot Q}\\
&= -\frac{k_R}{2}\tr{Q\hat e_\Omega + Q\hat \Omega_d-\hat\Omega_d Q}\\
&=-\frac{k_R}{2}\tr{Q\hat e_\Omega},
\end{align*}
where the last equality is obtained using $\tr{AB-BA}=0$ for any square matrices $A,B$.  From the identity $\tr{A\hat x} = - x\cdot (A-A^T)^\vee$ for any $x\in\Re^3$ and $A\in\Re^{3\times 3}$, the above is rewritten as
\[
\frac{k_R}{2} e_\Omega\cdot(Q-Q^T)^\vee = k_R e_R\cdot e_\Omega.
\]
Using this, along the trajectory of \refeqn{R_dot} and \refeqn{Omega:dot},
\begin{align*}
{\dot V}  (R,\Omega,t) &= \langle e_\Omega, k_R e_R - {\dot \Omega}_d + \MI^{-1}( \tau + (\MI \Omega) \times \Omega ) \rangle   \\
&+ \mu \langle C(R^TR_d) e_\Omega, e_\Omega \rangle \\ &+ \mu \langle e_R, \Omega_d \times e_\Omega - {\dot \Omega}_d + \MI^{-1}( \tau+ (\MI \Omega) \times \Omega ) \rangle,
\end{align*}
where the matrix $C(R^TR_d)$ is defined in (\ref{def:C:matrix}). Substituting the control (\ref{tracking:control}), and using the fact that $e_\Omega \times \Omega_d = (\Omega - \Omega_d) \times \Omega_d = \Omega \times \Omega_d$, this reduces to
\begin{align}\label{compute:Vdot}
{\dot V}  (R,\Omega,t) &= - k_\Omega \|e_\Omega\|^2 -\mu k_R\|e_R\|^2 - \mu k_\Omega \langle e_R, e_\Omega \rangle \nonumber \\
&\quad  + \mu \langle C(R^TR_d)e_\Omega, e_\Omega\rangle.
\end{align}
Setting $\mu=0$ yield \refeqn{V0_dot}. According to \cite{LeeLeoPICDC10}, the matrix $C(R^TR_d)$ defined in (\ref{def:C:matrix}) satisfies $\| C(R^TR_d)\|_2 \leq 1$, where $\| \cdot \|_2$ is the operator 2-norm. Using this fact, the Cauchy-Schwarz inequality and Lemma \ref{lemma:eR:ER}, one can easily prove (\ref{V:dot:lessthan:W3}). The given bound of $\mu$ guarantees the positive-definiteness of $W_3$.
\end{proof}
\end{lemma}

Next, we show that the proposed control system yields exponential stability.
\begin{theorem}\label{theorem:main}
Choose any positive numbers $k_R$, $k_\Omega$, $a$  and $\mu$ such that
\begin{equation}\label{ineq:a}
0 < a<1
\end{equation}
 and
\begin{equation}\label{ineq:epsilon}
0 < \mu < \frac{4(1-a)k_R k_\Omega}{4(1-a) k_R + k_\Omega^2}.
\end{equation}
Let
\begin{equation}\label{def:sigma}
\sigma = \frac{\lambda_{\min}(W_3)}{ \lambda_{\max}(W_2)} >0,
\end{equation}
where the matrices $W_2$ and $W_3$ are  defined in (\ref{def:W1}) and (\ref{def:W3}), respectively.

Then, the zero equilibrium of the tracking errors $(E_R,e_\Omega)=(0,0)$ is locally exponentially stable, and for any initial state $(R(0), \Omega(0)) \in \SO \times \mathbb R^3$ satisfying
\begin{equation}\label{V:ini}
V_0(R(0), \Omega(0), 0) \leq 2 a k_R,
\end{equation}
the closed-loop trajectory $(R(t),\Omega(t))$ for the control (\ref{tracking:control}) satisfies
\begin{gather}
V_0(R(t), \Omega(t), t) \leq V_0(R(0), \Omega(0), 0) \leq 2 a k_R, \label{V0:lessthan:akR}\\
V(R(t), \Omega(t), t) \leq V(R(0), \Omega(0), 0) e^{-\sigma t} \label{conclusion:theorem}
\end{gather}
for all $t\geq 0$. Furthermore, both the attitude tracking error $\|R(t) -R_d(t)\|$ and the body angular velocity tracking error $ \| \Omega(t) - \Omega_d(t)\|$ converge exponentially to zero at the exponential rate of $(-\sigma/2)$ as $t$ tends to infinity, i.e., \EditTL{there exists a  constant $c>0$ such that
\begin{equation}
\| E_R(t)\| +  \| e_\Omega(t) \|
\leq c\left( \| E_R(0)\| + \| e_\Omega(0) \|\right)  e^{-\frac{\sigma}{2} t}
\label{eqn:exp0}
\end{equation}
for  all $t\geq 0$ and  all initial state $(R(0), \Omega(0))$ satisfying (\ref{V:ini}).}

\begin{proof}
Take any positive numbers $k_R$, $k_\Omega$, $a$ and $\mu$ that satisfy (\ref{ineq:a}) and (\ref{ineq:epsilon}).   Since
\[
\frac{4(1-a)k_R k_\Omega}{4(1-a) k_R + k_\Omega^2} \leq \frac{4(1-a)k_R k_\Omega}{2\sqrt{4(1-a) k_R k_\Omega^2}} < \sqrt{k_R},
\]
we have
\[
0 < \mu < \sqrt{k_R}
\]
 by (\ref{ineq:epsilon}).  According to Lemma \ref{lemma:V:positive}, the matrix $W_1$ defined in (\ref{def:W1}) is positive-definite and both (\ref{W1:less:than:V}) and (\ref{V0:less:V}) hold for all $(R,\Omega) \in \SO \times \mathbb R^3$ and $t\geq 0$. It follows that the function  $V$ is positive-definite and decrescent.

Choose any initial state $(R(0),\Omega(0))$ satisfying (\ref{V:ini}).
By (\ref{eqn:V0_dot})  in Lemma \ref{lemma:Vdot:less},
$V_0(R(t),\Omega(t),t)$ is a non-increasing function of time along the closed-loop trajectory and (\ref{V0:lessthan:akR}) holds for all $t\geq 0$, which implies
\[
\|R(t) - R_d(t)\|^2 \leq \frac{4}{k_R}V_0(R(t), \Omega(t), t) \leq 8a
\]
for all $t\geq 0$. Since (\ref{eqn:ER2a}) holds, by Lemma \ref{lemma:Vdot:less} we have (\ref{V:dot:lessthan:W3}) with  $W_3$ being positive-definite.

The region of attraction to an asymptotically stable equilibrium is often estimated by a sub-level set of the Lyapunov function~\cite{Kha96}. While the estimate of the region of attraction given by (\ref{V:ini}) is not a sub-level set of the Lyapunov function $V(R,\Omega,t)$, it is positively invariant as $V_0(R(t),\Omega(t),t)$ is non-increasing in $t$. Therefore, for any initial condition satisfying (\ref{V:ini}), both of (\ref{W1:less:than:V}) and (\ref{V:dot:lessthan:W3}) hold true for all $t\geq 0$, and the exponential convergence is guaranteed as follows.

From (\ref{W1:less:than:V}) and (\ref{V:dot:lessthan:W3}), it follows
\begin{align*}
\dot V(R(t), \Omega (t), t)  
&\leq -\sigma V(R(t),\Omega(t),t)
\end{align*}
for all $t\geq 0$, where $\sigma$ is defined in (\ref{def:sigma}). This shows (\ref{conclusion:theorem}). Next, we show \refeqn{exp0}. From (\ref{V0:less:V}) and (\ref{conclusion:theorem}),
\[
V_0(R(t), \Omega(t), t) ) \leq \frac{\sqrt{k_R}+\mu}{\sqrt{k_R}-\mu} V_0(R(0), \Omega (0), 0) e^{-\sigma t},
\]
\EditTL{which implies
\begin{align*}
\frac{k_R}{4}&\| E_R(t)\|^2 + \frac{1}{2} \| e_\Omega(t) \|^2 \nonumber\\
&\leq \frac{\sqrt{k_R}+\mu}{\sqrt{k_R}-\mu}\left( \frac{k_R}{4}\| E_R(0)\|^2 + \frac{1}{2} \| e_\Omega(0) \|^2\right)  e^{-\sigma t}.
\end{align*}
We can view $\sqrt{\frac{k_R}{4}\|A\|^2 + \frac{1}{2}\|x\|^2}$ for $(A,x)\in\Re^{3\times 3}\times \Re^3$ as a norm on $\Re^{3\times 3}\times \Re^3$. As all norms on a finite-dimensional vector space are equivalent~\cite{Lan93},  there are positive constants $c_1$ and $c_2$ such that
\begin{equation}\label{norm:equivalence}
c_1 (\|A\| + \|x\|) \leq \sqrt{\frac{k_R}{4}\|A\|^2 + \frac{1}{2}\|x\|^2} \leq c_2 (\|A\| + \|x\|)
\end{equation}
for all $(A,x)\in\Re^{3\times 3}\times \Re^3$ . Hence, letting
\[
c = \frac{c_2}{c_1}\sqrt{\frac{\sqrt{k_R}+\mu}{\sqrt{k_R}-\mu}},
\]
we have  \refeqn{exp0} for all $t\geq 0$ and all initial states satisfying (\ref{V:ini}).}
\end{proof}
\end{theorem}

The following corollary characterizes the region of attraction that guarantees exponential stability, estimated by (\ref{V:ini}), and it discusses how to choose the values of the control parameters $k_R$ and $k_\Omega$ for a given initial state.
\begin{corollary}\label{corollary2:main}
Given an arbitrary initial state $(R(0), \Omega(0)) \in \SO \times \mathbb R^3$ satisfying
\begin{equation}
\| R(0) - R_d(0)\| < 2\sqrt 2,  \label{eqn:condR0}
\end{equation}
take any $k_R$ such that
\begin{equation}\label{Cor2:cond2}
\frac{2 \| \Omega (0) - \Omega_d(0)\|^2}{8 - \|R(0) - R_d(0)\|^2} < k_R.
\end{equation}
Then, $V_0(R(0), \Omega (0), 0) < 2k_R$, and the conclusion of Theorem \ref{theorem:main} holds true for any $k_\Omega$, $a$ and $\mu$ satisfying (\ref{ineq:epsilon}) and
\begin{equation}\label{Cor2:cond3}
\frac{V_0(R(0), \Omega (0),0)}{2k_R} \leq a <1.
\end{equation}

 \begin{proof}
 Straightforward.
 \end{proof}
 \end{corollary}

The inequalites \refeqn{condR0} and (\ref{Cor2:cond2}) imply that the proposed control system can handle any initial attitude excluding the set $\{ R\in \SO\mid\| R - R_d\| = 2\sqrt 2\}$, which is equal to $C_\pi R_d$, where $C_\pi$ is the conjugacy class containing rotations through angle $\pi$. Since $\dim (C_\pi  R_d) = \dim C_\pi = 2$ while $\dim\SO = 3$,  one can  claim that for a given $R_d \in \SO$, the set $\{R \in \SO \mid \|R-R_d\| < 2\sqrt 2\}$ almost covers the entire space $\SO$.  Therefore,  Corollary \ref{corollary2:main} implies that our tracking control law can handle a large set of initial attitude tracking errors, excluding a set of  measure zero only. This property is referred to as \textit{almost} global exponential stability~\cite{Kod88}, and it is considered as the strongest stability property for smooth attitude controls, due to the topological obstruction that prohibits global attractivity with smooth vector fields on $\SO$~\cite{BhaBerSCL00}.

While there have been attitude controllers achieving {almost} global exponential stability~\cite{LeeITAC15,LeeSCL12}, they showed the exponential convergence for the auxiliary attitude error vector $e_R$, which is not necessarily proportional to the attitude error~\cite{LeeSCL12}. Instead, the stability analysis presented in this section guarantees the exponential convergence of the attitude error $E_R$ satisfying $\|E_R\|=\|R-R_d\|=\|I_{3\times 3}-R^T R_d\|$.

\subsection{Global Tracking Strategy}\label{sec:SGT}

In practice, there is a limited chance that the initial attitude is placed in the low-dimensional set of $C_\pi  R_d$ that does not guarantee the exponential convergence. However, it is shown that the controlled trajectories may be strongly affected by the existence of the stable manifolds of points in  $C_\pi  R_d$, and the rate of convergence can be reduced significantly~\cite{LeeLeoPICDC11,LeeSCL12}.

To avoid these issues, hybrid attitude control systems have been introduced to achieve global asymptotic stability. These are based on a class of attitude error functions, referred to as synergistic potential functions, that are constructed by stretching and scaling the popular trace form of the attitude error function~\cite{MayTeeITAC13}, or by expelling the controlled attitude trajectories away from the undesired equilibria~\cite{LeeITAC15}. In these approaches, the topological obstruction to global attractivity is avoided by using discontinuities of the control input with respect to time. However, attitude actuators are constrained by limited bandwidth, and abrupt changes in the control torque may excite the unmodeled dynamics and cause undesired behaviors, such as the vibrations of solar panels in satellites. Interestingly, it has not been known whether a more general class of time-varying feedback could accomplish the global stabilization task without introducing such discontinuities~\cite{MayTeeA13}.

In this section, we present an alternative approach to achieve global attractivity. The key idea is to introduce a \textit{shifted} desired attitude, and design an attitude control system to follow the shifted trajectories instead of the original attitude command. The shifted desired trajectory is carefully constructed with the conjugacy class discussed in Section \ref{sec:PF} to guarantee the convergence to the desired attitude trajectory from any initial attitude. In contrast to hybrid attitude controls where the configuration error function defining the controlled dynamics is switched instantaneously, the proposed approach adjusts the desired attitude trajectory continuously in time, thereby avoiding any jump in the control input.

Consider the attitude control system presented in Theorem \ref{theorem:main}. Suppose the initial condition $(R(0),\Omega(0))$ satisfies (\ref{V:ini}). Then, the exponential convergence is guaranteed, and there is no need for modification. As such, this section focuses on the other case when the initial condition does not satisfy the given estimate of the region of attraction, i.e.,
\begin{equation}\label{eqn:V0_GT}
V_0(R(0), \Omega(0), 0) > 2 a k_R.
\end{equation}

Now, we introduce the \textit{shifted} desired attitude. By (\ref{R:isometry}), (\ref{Z:distance}), and (\ref{C:union}), there exists a unique $\theta_0 \in [0, \pi]$ such that $R(0)R_d(0)^T \in C_{\theta_0}$, i.e.
\begin{equation}
R(0) = U_0 Z_{\theta_0}U_0^T R_d(0)\label{eqn:R0:U:thetab}
\end{equation}
for some $U_0 \in \SO$. In other words, the initial attitude and the initial desired attitude is related by the fixed-axis rotation by the angle $\theta_0$ about the third column of $U_0$.

Next, for a constant $\epsilon\in(0,1)$, choose an angle $\theta_{b_0} \in (0, \theta_0) $ such that
\begin{equation}\label{eqn:newly:added:theta}
1-\cos(\theta_0 - \theta_{b_0}) \leq 2\epsilon a.
\end{equation}
Note that such an angle $\theta_{b_0}$ always exists for any  $\epsilon,a\in(0,1)$ since one has only to make $|\theta_0 - \theta_{b_0}|$ sufficiently small. Then, define a time-varying angle $\theta_b(t)$ as
\begin{equation}
\theta_b(t) = \theta_{b_0}e^{-\frac{\gamma}{2} t},\label{eqn:theta:b:varying}
\end{equation}
where $\gamma >0$ is a positive constant satisfying
\begin{equation}\label{eqn:gamma}
\gamma < \frac{4}{\theta_{b_0}}\sqrt{a k_R (1-\epsilon)}.
\end{equation}
Using $\theta_b(t)$, define the shifted desired attitude as
\begin{equation}
\tilde R_d(t) = U_0 Z_{\theta_b(t)}U_0^T R_d(t).\label{eqn:Rd:shifted:varying}
\end{equation}
Therefore, $R_d(t)^T \tilde R_d(t)$ belongs to the conjugacy class $C_{\theta_b(t)}$. The properties of the shifted desired attitude are summarized as follows.

\begin{lemma}\label{lem:Rdtilde}
Consider the shifted desired attitude trajectory given by \refeqn{Rd:shifted:varying}.
\begin{itemize}
\item[(i)] The initial attitude error from the shifted desired trajectory satisfies
\begin{equation}
\|R(0)-\tilde R_d(0)\| = 2\sqrt{1-\cos(\theta_0-\theta_{b_0})}< 2\sqrt{2a}.\label{RminusRtilde}
\end{equation}
\item[(ii)] The difference between the shifted desired attitude and the true desired attitude exponentially converges to zero as
\begin{equation}\label{eqn:Rdtilde:Rd}
\|\tilde R_d(t) - R_d(t)\|  = 2\sqrt{1-\cos\theta_b(t)} \leq \sqrt{2} \theta_{b_0} e^{-\frac{\gamma}{2} t}.
\end{equation}
\item[(iii)] The time-derivative of the shifted desired attitude trajectory is given by
\begin{align}
\dot {\tilde R}_d(t) = \tilde R_d(t) \hat{\tilde\Omega}_d(t),
\end{align}
where the shifted desired angular velocity $\tilde\Omega_d(t)$ is
\begin{equation}\label{eqn:Omega:tilde:varying}
\tilde\Omega_d(t) = \Omega_d(t) + \dot\theta_b(t) \tilde R_d(t)^T U_0 e_3.
\end{equation}

\item[(iv)] The difference between the shifted desired angular velocity and the true desired angular velocity is given by
\begin{equation}\label{eqn:Wdtilde:Wd}
\|\tilde\Omega_d(t)-\Omega_d(t)\| = |\dot\theta_b(t)|
=\frac{\gamma}{2}\theta_{b_0} e^{-\frac{\gamma}{2} t}.
\end{equation}

\end{itemize}

\begin{proof}
From (\ref{R:isometry}) and (\ref{Z:distance}),
\begin{align*}
\|R(0)-\tilde R_d(0)\|& =
\| U_0 Z_{\theta_0} U_0^T R_d(0)-U_0 Z_{\theta_{b_0}} U_0^T R_d(0)\|\\
&= \|Z_{\theta_0}-Z_{\theta_{b_0}}\| = \sqrt{1-\cos(\theta_0-\theta_{b_0})},
\end{align*}
which shows (\ref{RminusRtilde}) from \refeqn{newly:added:theta}. Similarly,
\[
\|\tilde R_d(t) - R_d(t)\| = \|Z_{\theta_b(t)} - I_{3\times 3}\| = 2\sqrt{1-\cos\theta_b(t)}.
\]
From the fact that $\frac{2}{\pi^2}x^2 \leq 1-\cos x \leq \frac{1}{2}x^2$ for any $x\in[0,\pi]$,  the last inequality of \refeqn{Rdtilde:Rd} follows.

Next, the time-derivative of the shifted reference attitude is
\[
\frac{d}{dt}\tilde R_d(t) = \dot\theta_b(t) U_0 \hat e_3 Z_{\theta_b(t)}  U_0^T R_d(t)+\tilde R_d(t)\hat \Omega_d(t),
\]
which can be  rewritten as
\[
\frac{d}{dt}\tilde R_d(t) = \dot\theta_b(t) U_0 \hat e_3 U_0^T \tilde R_d(t)+\tilde R_d(t)\hat \Omega_d(t)
\]
by \refeqn{Rd:shifted:varying}.
Using the property $\widehat {Rx}=R\hat x R^T$ for any $R\in\SO$ and $x\in\Re^3$ repeatedly,
\begin{align*}
\frac{d}{dt}\tilde R_d(t) &= \dot\theta_b(t)  \widehat {U_0e_3} \tilde R_d(t)+\tilde R_d(t)\hat \Omega_d(t)\\
& = \dot\theta_b(t) \tilde R_d(t) (\tilde R_d(t)^TU_0e_3)^\wedge+\tilde R_d(t)\hat \Omega_d(t),
\end{align*}
and this shows \refeqn{Omega:tilde:varying}. It is straightforward to show \refeqn{Wdtilde:Wd} as $\|\tilde R_d(t)^T U_0 e_3\|=1$ for any $t\geq 0$.
\end{proof}
\end{lemma}

The motivation for the proposed shifted desired attitude is that the initial attitude error, namely
\begin{equation*}
\|E_R(0)\|=\|R(0)-R_d(0)\| = \| Z_{\theta_0} - I_{3\times 3} \| = 2\sqrt{1-\cos\theta_0},
\end{equation*}
is replaced by the shifted error
\[
\|R(0)-\tilde R_d(0)\|=2\sqrt{1-\cos(\theta_0-\theta_{b_0})},
\]
that is strictly less than $2\sqrt{2a}$ from (\ref{RminusRtilde}). In other words, the inequality \refeqn{newly:added:theta} ensures that the initial value of the shifted desired attitude, namely $\tilde R_d(0)$ is sufficiently close to the initial attitude $R(0)$. As such, even when the initial attitude error $\|R(0)-R_d(0)\|$ is close or equal to $2\sqrt{2}$, we can replace  it with the shifted desired attitude so as to satisfy \refeqn{condR0}. Furthermore, as shown by \refeqn{Rdtilde:Rd} and \refeqn{Wdtilde:Wd}, the shifted desired trajectories $(\tilde R_d(t),\tilde\Omega_d(t))$ exponentially converge to their true desired trajectories $(R_d(t),\Omega_d(t))$ as $t$ tends to infinity. This is because $\theta_b(t)$ exponentially converges to zero from \refeqn{theta:b:varying}. Therefore, we can design a control system to follow the shifted desired trajectories, while ensuring asymptotic convergence to the true desired trajectories.

More explicitly, the shifted attitude error variables $\tilde E_R\in\Re^{3\times 3}$ and $\tilde e_R\in\Re^3$ are defined as
\begin{align}
\tilde E_R &= R - \tilde R_d,\label{eqn:ER:tilde}\\
\tilde e_R &= \frac{1}{2} (\tilde R_d^T R - R^T\tilde R_d)^\vee.\label{eqn:eR:tilde}
\end{align}
Also, the shifted angular velocity error is defined as
\begin{equation}\label{eqn:eW:tilde}
\tilde e_\Omega = \Omega -\tilde\Omega_d\in\Re^3.
\end{equation}
As with (\ref{tracking:control}), the control input for the shifted desired reference can be designed as
\begin{equation}\label{def:tilde:tau:varying}
\tilde \tau =   - (\MI \Omega)\times \Omega + \MI ( - k_R \tilde e_R -  k_\Omega \tilde e_\Omega + \Omega \times \tilde\Omega_d + \dot{\tilde \Omega}_d).
\end{equation}
From Theorem \ref{theorem:main}, for any initial condition satisfying
\begin{equation}\label{eqn:ROA_2}
\frac{k_R}{4}\|R(0)-\tilde R_d(0)\|^2 + \frac{1}{2}\|\Omega(0)-\tilde \Omega_d(0)\|^2 \leq 2a k_R,
\end{equation}
which is equivalent to (\ref{V:ini}) for the shifted desired trajectory, the trajectory of the controlled system exponentially converges to the shifted desired trajectories $(\tilde R_d(t),\tilde\Omega_d(t))$ that tends to the true desired trajectories $(R_d(t),\Omega_d(t))$. The resulting stability properties are summarized as follows.

\begin{theorem}\label{thm:main_shifted_varying}
Choose any positive constants $k_R$, $k_\Omega$, $a$, $\mu$, $\sigma$, $\epsilon$, $\theta_{b_0}$, and $\gamma$ such that (\ref{ineq:a}), (\ref{ineq:epsilon}), (\ref{def:sigma}), \refeqn{newly:added:theta} and \refeqn{gamma} are satisfied. The control input is defined as
\begin{subnumcases}{\label{eqn:tau_3} \tau =}
 - (\MI \Omega)\times \Omega + \MI ( - k_R e_R -  k_\Omega e_\Omega + \Omega \times \Omega_d + \dot \Omega_d)\nonumber\\
  \qquad\text{when } V_0(R(0),\Omega(0),0)\leq 2ak_R,\label{eqn:tau_3a}\\
- (\MI \Omega)\times \Omega + \MI ( - k_R \tilde e_R -  k_\Omega \tilde e_\Omega + \Omega \times \tilde\Omega_d + \dot{\tilde \Omega}_d)\nonumber\\
  \qquad \text{when } V_0(R(0),\Omega(0),0)> 2ak_R,\label{eqn:tau_3b}
\end{subnumcases}
where $\tilde e_R$ and $\tilde e_\Omega$ are constructed by \refeqn{eR:tilde} and \refeqn{eW:tilde}, respectively.

Then, the zero equilibrium of the tracking errors $(E_R,e_\Omega)=(0,0)$ is exponentially stable. More specifically, when $V_0(R(0),\Omega(0),0)\leq 2ak_R$, the tracking errors $(E_R(t),e_\Omega(t))$ converge to zero exponentially according to \refeqn{exp0}. \EditTL{Otherwise, when $V_0(R(0),\Omega(0),0)> 2ak_R$, there exists a positive constant $c>0$ such that, for any initial state $(R(0), \Omega(0)) \in \SO \times \mathbb R^3$ satisfying \refeqn{ROA_2},
\begin{align}
\|E_R(t)\|+\|e_\Omega(t)\| \leq c(\|E_R(0)\|+\|e_\Omega(0)\|)\exp^{-\frac{1}{2}\min\{\sigma,\gamma\}t},\label{eqn:exp2}
\end{align}
for all $t\geq 0$.}

\end{theorem}

\begin{proof}
When $V_0(R(0),\Omega(0),0)\leq 2ak_R$, the results of Theorem \ref{theorem:main} are directly applied, i.e., the tracking errors $(E_R(t),e_\Omega(t))$ converge to zero exponentially according to \refeqn{exp0}.

Next, when $V_0(R(0),\Omega(0),0)> 2ak_R$, for any initial condition satisfying \refeqn{ROA_2}, according to Theorem \ref{theorem:main}, there exists $\tilde c>0$ such that
\begin{equation}\label{eqn:tildez}
\|\tilde E_R(t)\|+\|\tilde e_\Omega(t)\| \leq \tilde c (\|\tilde E_R(0)\|+\|\tilde e_\Omega(0)\|)e^{-\frac{\sigma}{2}t}
\end{equation}
for all $t\geq 0$. From the triangle inequality,
\begin{align*}
\|E_R(t)\|&+\|e_\Omega(t)\| \leq \|\tilde E_R(t)\|+\|\tilde e_\Omega(t)\|\nonumber\\
& +
\|\tilde R_d(t)-R_d(t)\|+\|\tilde \Omega_d(t)-\Omega_d(t)\|.
\end{align*}
Substituting \refeqn{Rdtilde:Rd}, \refeqn{Wdtilde:Wd}, and \refeqn{tildez},
\begin{align}
\|E_R(t)\|+\|e_\Omega(t)\| &\leq \tilde c (\|\tilde E_R(0)\|+\|\tilde e_\Omega(0)\|)e^{-\frac{\sigma}{2}t}\nonumber\\
&\quad+\parenth{\sqrt{2}+\frac{\gamma}{2}} \theta_{b_0} e^{-\frac{\gamma}{2} t}.
\label{eqn:z0}
\end{align}

Now, we derive several inequalities to show \refeqn{exp2}.
The initial shifted attitude error satisfies
\begin{align*}
\|\tilde E_R(0)\| &\leq \|E_R(0)\|+\|R_d(0)-\tilde R_d(0)\|\\
&= \|E_R(0)\| + 2\sqrt{1-\cos\theta_b(0)}.
\end{align*}
Since $\theta_b(0)=\theta_{b_0} < \theta_0$ from the definition of $\theta_{b_0}$,
\begin{equation}
\|\tilde E_R(0)\|< \|E_R(0)\| + 2\sqrt{1-\cos\theta_0} = 2\|E_R(0)\|.\label{eqn:ER:tilde:0}
\end{equation}
 Also, from \refeqn{Wdtilde:Wd},
\begin{align}\label{eqn:tilde:eW:bound}
\|\tilde e_\Omega(0)\|&\leq \|e_\Omega(0)\| + \|\Omega_d(0)-\tilde\Omega_d(0)\|\nonumber\\
&\leq \|e_\Omega(0)\| + \frac{\gamma}{2}\theta_{b_0}.
\end{align}
Next, from the fact that $\frac{2}{\pi^2}x^2 \leq 1-\cos x \leq \frac{1}{2}x^2$ for any $x\in[0,\pi]$ and $\theta_{b_0}<\theta_0$, we have
\begin{align*}
\theta_{b_0} & \leq \frac{\pi}{\sqrt{2}}\sqrt{1-\cos\theta_{b_0}}<\frac{\pi}{\sqrt{2}}\sqrt{1-\cos\theta_{0}}=\frac{\pi}{2\sqrt{2}}\|E_R(0)\|,
\end{align*}
which is substituted into \refeqn{z0} together with \refeqn{ER:tilde:0} to obtain
\begin{align*}
&\|E_R(t)\|+\|e_\Omega(t)\| \nonumber\\
&\leq \braces{\parenth{2\tilde c + \frac{\pi}{2\sqrt{2}} \parenth{\frac{\gamma}{2}(1+\tilde c)+\sqrt{2}}}\|E_R(0)\|+\tilde c\|e_\Omega(0)\|}\\
&\quad \times e^{-\frac{1}{2}\min\{\sigma,\gamma\}t}.
\end{align*}
This shows \refeqn{exp2}, which guarantees exponential stability.
\end{proof}

Next, we characterize the region of attraction of the proposed control system as follows.
\begin{corollary}\label{cor:ROA_GT}
For the control system presented in Theorem \ref{thm:main_shifted_varying}, the region of attraction guaranteeing the exponential convergence encloses the following set,
\begin{align}\label{eqn:ROA}
&\mathcal{R}=\{(R(0),\Omega(0))\in\SO\times\Re^3\,|\, \|e_\Omega(0)\| <\nonumber \\
&\max\{
\sqrt{2k_R(2a-1+\cos\theta_0)},\,2\sqrt{ak_R(1-\epsilon)}-\frac{\gamma}{2}\theta_{b_0}\} \},
\end{align}
where $\theta_0$ is constructed from $R(0)$ by \refeqn{R0:U:thetab}.

Furthermore, $\mathcal{R}\subset\SO\times\Re^3$ covers $\SO$ completely, and when projected onto $\Re^3$, it is enlarged into $\Re^3$ as $k_R$ is increased in a semi-global sense.
\end{corollary}

\begin{proof}
Define three subsets of $\SO\times\Re^3$ for the initial condition as
\begin{align*}
\mathcal{R}_1&=\braces{k_R(1-\cos\theta_0)+\frac{1}{2}\|e_\Omega(0)\|^2\leq 2ak_R },\\
\mathcal{R}_2&=\braces{k_R(1-\cos\theta_0)+\frac{1}{2}\|e_\Omega(0)\|^2> 2ak_R },\\
\mathcal{R}_3&=\braces{k_R(1-\cos(\theta_0-\theta_{b_0}))+\frac{1}{2}\|\tilde e_\Omega(0)\|^2\leq 2ak_R }.
\end{align*}
The sets $\mathcal{R}_1$ and $\mathcal{R}_2$ represent the initial conditions corresponding to the two cases of the control inputs, namely \refeqn{tau_3a} and \refeqn{tau_3b}, respectively. The set $\mathcal{R}_3$ represents the set of initial conditions \refeqn{ROA_2}, guaranteeing the exponential convergence for the second case of the control input. Therefore, the combined region of attraction $\bar{\mathcal{R}}$ guaranteeing exponential stability is given by $\bar{\mathcal{R}}\equiv\mathcal{R}_1\cup(\mathcal{R}_2\cap\mathcal{R}_3)$. Since $\mathcal{R}_1\cup\mathcal{R}_2=\SO\times\Re^3$, this reduces to $\bar{\mathcal{R}}=(\mathcal{R}_1\cup\mathcal{R}_2)\cap
(\mathcal{R}_1\cup\mathcal{R}_3)=\mathcal{R}_1\cup\mathcal{R}_3$.

Now, we show the set $\mathcal{R}$ defined by \refeqn{ROA} is contained in $\bar{\mathcal{R}}$, i.e., $\mathcal{R}\subset\bar{\mathcal{R}}$. For any $(R(0),\Omega(0))\in\mathcal{R}$,
\[
\|e_\Omega(0)\| < \sqrt{2k_R(2a-1+\cos\theta_0)},
\]
or
\[
\|e_\Omega(0)\| < 2\sqrt{ak_R(1-\epsilon)}-\frac{\gamma}{2}\theta_{b_0},
\]
where the right-hand side is positive due to \refeqn{gamma}. For the former case, it is straightforward to show $(R(0),\Omega(0))\in\mathcal{R}_1\subset\bar{\mathcal{R}}$. For the latter case, from \refeqn{tilde:eW:bound},
\[
\|\tilde e_\Omega(0)\| \leq 2\sqrt{ak_R(1-\epsilon)}.
\]
Therefore,
\begin{align*}
k_R& (1-\cos(\theta_0-\theta_{b_0}))+\frac{1}{2}\|\tilde e_\Omega(0)\|^2 \\
&\leq
k_R(1-\cos(\theta_0-\theta_{b_0})) + 2a(1-\epsilon)k_R \leq 2ak_R,
\end{align*}
where the last inequality is obtained by \refeqn{newly:added:theta}. This follows $(R(0),e_\Omega(0))\in\mathcal{R}_3\subset\bar{\mathcal{R}}$. In short, any initial condition in $\mathcal{R}$ belongs to the estimated region of attraction $\bar{\mathcal{R}}=\mathcal{R}_1\cup \mathcal{R}_3$ for the controlled dynamics, and the corresponding trajectory converges to zero exponentially according to \refeqn{exp0} or \refeqn{exp2}.

Next, for any $\theta_0\in[0,\pi]$, the set $\mathcal{R}$ is non-empty due to \refeqn{gamma}. As such, $\mathcal{R}$ contains $\SO$ via (\ref{C:union}). At \refeqn{ROA}, the upper bound of $\|e_\Omega(0)\|$ tends to be infinite, as $k_R\rightarrow\infty$. Therefore, $\mathcal{R}$ covers $\Re^3$ in a semi-global sense, as $k_R\rightarrow\infty$.
\end{proof}

The exceptional  property of the proposed control system is that the region of attraction covers the special orthogonal group completely, but the control input formulated in \refeqn{tau_3} is continuous in the time, i.e., there is no switching through the controlled attitude dynamics. The essential idea is that when the initial errors are large, the desired attitude is altered to an attitude that is closer to the initial attitude along the same conjugacy class, which is gradually varied back to the true desired attitude.

Taking the advantage of the global attractivity in attitude controls with non-switching controls has been unprecedented, and it has been considered largely impossible to achieve that. As discussed above, it has been uncertain if a more general class of time-varying feedback could accomplish the global stabilization task without introducing switching~\cite{MayTeeA13}. The proposed control system overcomes the topological restriction with  the discontinuity of the control input with respect to the initial condition. The proposed framework of modifying the desired trajectory to achieve global attractivity has been unprecedented, and it is readily generalized to abstract Lie groups and homogeneous manifolds.

\section{Adaptive Attitude Tracking Controls}\label{sec:AATC}

In this section, a constant disturbance moment $\Delta$ that has been introduced in \refeqn{Omega:dot} is considered. \EditTL{Throughout this section, it is assumed that the reference trajectory is such that both $\Omega_d(t)$ and $\dot \Omega_d(t)$ are bounded.} The organization of this section is parallel to the preceding section: two types of attitude tracking strategies are presented with an adaptive law to eliminate the effects of the disturbance.

\subsection{Almost Global Adaptive Tracking Strategy}

The overall controller structure and the definition of the error variables and parameters are identical to those defined in Section \ref{sec:AGT}. The adaptive control law presented in this section includes an estimate of the disturbance, denoted by $\bar\Delta\in\Re^3$ in the control torque as
\begin{equation}\label{eqn:tracking:control:delta}
\tau =  - (\MI \Omega)\times \Omega + \MI ( -k_R e_R - k_\Omega e_\Omega + \Omega \times \Omega_d + \dot \Omega_d)-\bar\Delta,
\end{equation}
where $\bar\Delta$ is updated according to
\begin{equation}\label{eqn:bar:delta:dot}
\dot{\bar \Delta} = k_\Delta \MI^{-1} (e_\Omega + \mu e_R)
\end{equation}
for $k_\Delta > 0$ with the initial estimate $\bar\Delta(0)=0$.

Let the estimation error be
\[
e_\Delta = \Delta -\bar \Delta\in\Re^3.
\]
From Assumption \ref{assump:delta}, we have the bound for the initial estimation error as $\|e_\Delta(0)\|\leq \delta$.

Define a Lyapunov function, augmented with an additional term for the estimation error $e_\Delta$ as
\begin{equation}\label{eqn:V:bar}
\bar V(R(t),\Omega(t),\bar\Delta(t),t) = V(R(t),\Omega(t),t) + \frac{1}{2k_\Delta} \| e_\Delta(t)\|^2.
\end{equation}
Along the trajectory of the controlled system with \refeqn{tracking:control:delta},
\begin{align*}
\dot{\bar V}(t) & =  - k_\Omega \|e_\Omega\|^2 -\mu k_R\|e_R\|^2 - \mu k_\Omega \langle e_R, e_\Omega \rangle \nonumber \\
&\,  + \mu \langle C(R^TR_d)e_\Omega, e_\Omega\rangle
+\langle e_\Delta, \MI^{-1} (e_\Omega+\mu e_R) -\frac{1}{k_\Delta} \dot{\bar \Delta}\rangle,
\end{align*}
where $\bar V(t)$ is a shorthand for $\bar V(R(t),\Omega(t),\bar\Delta(t),t)$. Substituting \refeqn{bar:delta:dot}, it reduces to
\begin{align}
\dot{\bar V}(t) & = - k_\Omega \|e_\Omega\|^2 -\mu k_R\|e_R\|^2 - \mu k_\Omega \langle e_R, e_\Omega \rangle \nonumber \\
&\,\, + \mu \langle C(R^TR_d)e_\Omega, e_\Omega\rangle.\label{eqn:Vbar_dot0}
\end{align}
Thus, the proposed control torque \refeqn{tracking:control:delta} and the adaptive law \refeqn{bar:delta:dot} ensure that the time-derivative of the augmented Lyapunov function is identical to (\ref{compute:Vdot}) that is developed for the ideal case when $\Delta=0$. The corresponding stability properties are summarized as follows.

\begin{theorem}\label{theorem:main:robust}
Consider the control torque defined in \refeqn{tracking:control:delta} with the adaptive law \refeqn{bar:delta:dot}. Choose positive constants $k_R$, $k_\Omega$, $k_\Delta$, $a$, and $\mu$ such that (\ref{ineq:a}), (\ref{ineq:epsilon}), and the following inequality are satisfied.
\begin{equation}\label{eqn:kdelta}
0 < 2a\frac{\sqrt{k_R}-\mu}{\sqrt{k_R}+\mu}k_R-\frac{1}{2k_\Delta}\delta^2.
\end{equation}
Then, the desired reference trajectory $(R(t),\Omega(t),\bar\Delta(t))=(R_d(t),\Omega_d(t),\Delta)$ is asymptotically stable. Furthermore, for any initial state $(R(0), \Omega(0)) \in \SO \times \mathbb R^3$ satisfying
\begin{align}\label{eqn:ineq:var:V0}
V_0(R(0),\Omega(0),0)\leq 2a\frac{\sqrt{k_R}-\mu}{\sqrt{k_R}+\mu}k_R-\frac{1}{2k_\Delta}\delta^2,
\end{align}
all of the attitude tracking error $\|R(t)-R_d(t)\|$, the angular velocity tracking error $\|\Omega(t)-\Omega_d(t)\|$ and the estimation error $\|\Delta-\bar\Delta(t)\|$ converge to zero as $t$ tends to infinity.
\end{theorem}

\begin{proof}
According to the proof of Theorem \ref{theorem:main}, the augmented Lyapunov function \refeqn{V:bar} satisfies
\begin{equation}\label{eqn:V4}
z^T W_1 z + \frac{1}{2k_\Delta}\|e_\Delta\|^2 \leq \bar V(t) \leq
z^T W_2 z + \frac{1}{2k_\Delta}\|e_\Delta\|^2,
\end{equation}
for all $t\geq 0$, where the matrices $W_1$ and $W_2$ defined in (\ref{def:W1}) are positive-definite as $\mu < \sqrt{k_R}$ from (\ref{ineq:a}) and (\ref{ineq:epsilon}).

Since $\|C(R^T R_d)\|_2\leq 1$, from \refeqn{Vbar_dot0},
\[
\dot{\bar V}(t) \leq - (k_\Omega-\mu) \|e_\Omega\|^2 -\mu k_R\|e_R\|^2 + \mu k_\Omega \|e_R\| \|e_\Omega\|,
\]
and it is negative-semidefinite, i.e., $\dot{\bar V} \leq 0$ as $\mu < \frac{4k_Rk_\Omega}{4k_R+k_\Omega^2}$ from (\ref{ineq:epsilon}). Hence $\bar V(t)$ is non-increasing.

The condition \refeqn{kdelta} ensures that the set of initial conditions satisfying \refeqn{ineq:var:V0} is a non-empty neighborhood of the desired reference trajectory. From (\ref{V0:less:V}) and the fact that $\|e_\Delta(0)\|\leq \delta$,
\begin{align*}
\bar V & (R(0),\Omega(0),\bar\Delta(0),0)\nonumber\\
& \leq \frac{\sqrt{k_R}+\mu}{\sqrt{k_R}} V_0(R(0),\Omega(0),0)+\frac{1}{2k_\Delta}\delta^2.
\end{align*}
Therefore, for any initial condition satisfying \refeqn{ineq:var:V0},
\begin{equation}
\bar V  (R(0),\Omega(0),\bar\Delta(0),0) \leq 2(k_R-\mu\sqrt{k_R})a. \label{eqn:ROA_3a}
\end{equation}
Since $\bar V(t)$ is a non-increasing function of time, using the lower bound of (\ref{V0:less:V}), the above inequality implies
\begin{align*}
& \|R(t)-R_d(t)\|^2 \leq \frac{4}{k_R}V_0(t)\leq \frac{4}{k_R -\mu\sqrt{k_R}} V(t)\nonumber\\
  &\leq \frac{4}{k_R -\mu\sqrt{k_R}} \bar V(t)
  \leq \frac{4}{k_R -\mu\sqrt{k_R}} \bar V(0) \leq 8a.
\end{align*}
Therefore, \refeqn{ER2a} is satisfied. As discussed above, the expression for $\dot{\bar V}(t)$ is identical to (\ref{compute:Vdot}), and therefore, we can apply Lemma \ref{lemma:Vdot:less} to obtain
\begin{equation}
 \dot{\bar V}(t) \leq - z^T W_3 z,\label{eqn:V4_dot}
 \end{equation}
for all $t\geq 0$ with $z=(\|E_R\|,\|e_\Omega\|)$, and $W_3$ is positive-definite due to (\ref{ineq:epsilon}).

In short, for any initial state $(R(0), \Omega(0)) \in \SO \times \mathbb R^3$ satisfying
\refeqn{ineq:var:V0}, the Lyapunov function is positive-definite and decrescent as \refeqn{V4}, and its time-derivative is negative-semidefinite as \refeqn{V4_dot}, for all $t\geq 0$. This implies that the reference trajectory is stable in the sense of Lyapunov, and all of the error variables,  namely $E_R(t)$, $e_\Omega(t)$, and $e_\Delta(t)$ are  bounded. Also, from the LaSalle-Yoshizawa theorem~\cite[Theorem A.8]{KrsKan95}, $z = (\| E_R \|, \| e_\Omega \|)\rightarrow 0$ as $t\rightarrow\infty$. Using the boundedness of the error variables, one can easily show that $\ddot e_\Omega$ is bounded as well. Then, according to Barbalat's lemma, $\dot e_\Omega \rightarrow 0$ as $t\rightarrow\infty$.  Substituting these into \refeqn{Omega:dot} and \refeqn{tracking:control:delta} guarantees $\bar\Delta\rightarrow \Delta$ as $t\rightarrow\infty$. Therefore, the zero equilibrium of the tracking errors and the estimation error is asymptotically stable.
\end{proof}

In contrast to Theorem \ref{theorem:main}, there is no guarantee of exponential convergence as the  given ultimate bound of $\dot{\bar V}$ does not depend on the estimation error $e_\Delta$. Alternatively, one could achieve exponential stability by redefining the error variables as shown in~\cite{LeeITAC15}, but we do not pursue it in this paper.

An alternative estimate of the region of attraction is given by the sub-level set of the Lyapunov function as \refeqn{ROA_3a}. The estimate provided by \refeqn{ineq:var:V0} in Theorem \ref{theorem:main:robust} is more conservative, as it is a subset of \refeqn{ROA_3a}. However, we use \refeqn{ineq:var:V0} as an estimate of the region of attraction in the subsequent development throughout this section for simplicity. As with Corollary \ref{corollary2:main}, we discuss how to choose  values of the control parameters $k_R$ and $k_\Omega$ for a given initial state.
\begin{corollary}\label{corollary2:main:robust}
Given an arbitrary initial state $(R(0), \Omega(0),\bar\Delta(0)) \in \SO \times \mathbb R^3\times \Re^3$ satisfying
\[
\| R(0) - R_d(0)\| < 2\sqrt{2},
\]
take any $k_R$, $k_\Delta$, $k_\Omega$, and $\mu$ such that
\begin{gather}\label{Cor2:cond42}
\| R(0) - R_d(0)\|^2 < 8\frac{\sqrt{k_R}-\mu}{\sqrt{k_R}+\mu},\\
\frac{2 \| \Omega (0) - \Omega_d(0)\|^2+\frac{2}{k_\Delta}\delta^2}{8\frac{\sqrt{k_R}-\mu}{\sqrt{k_R}+\mu} - \|R(0) - R_d(0)\|^2} < k_R
\end{gather}
with (\ref{ineq:epsilon}) and \refeqn{kdelta}.  Then, $V_0(R(0),\Omega(0),0)< 2a \frac{\sqrt{k_R}-\mu}{\sqrt{k_R}+\mu}-\frac{1}{2k_\Delta}\delta^2$, and the conclusions of Theorem \ref{theorem:main:robust} hold true for any $a$ satisfying
\[
\frac{V_0(R(0),\Omega(0),0)+\frac{1}{2k_\Delta}\delta^2}{2\frac{\sqrt{k_R}-\mu}{\sqrt{k_R}+\mu}} \leq a <1.
\]

 \begin{proof}
 Straightforward.
 \end{proof}
 \end{corollary}

This corollary implies that the region of attraction for asymptotic convergence almost covers $\SO$ as discussed in Section \ref{sec:AGT}, and the region of attraction for $\Omega$ increases in a semi-global sense by increasing $k_R$.

While these results are developed for the constant disturbance, it can be readily generalized to any other disturbance models where the unknown parameter appears linearly, such as $\phi(R,\Omega,t)\Delta$ with a known function $\phi(R,\Omega,t)$, as commonly studied in the literature of adaptive controls with a weaker convergence property that $\phi(R,\Omega,t)e_\Delta$ asymptotically converges to zero~\cite{AstWit08,IoaSun95}. The  result presented in this paper may be considered as a special case when $\phi(R,\Omega,t)=I_{3\times 3}$.

\subsection{Adaptive Global Tracking Strategy}\label{sec:RSGT}

As in Section \ref{sec:SGT}, we construct an adaptive control law that is continuous in time, while guaranteeing global attractivity. Consider the adaptive control system presented in Theorem \ref{theorem:main:robust}. If a given initial condition satisfies \refeqn{ineq:var:V0}, the error variables asymptotically converge to zero. Therefore, we focus on the case where
\begin{align}\label{eqn:ineq:var:V0:a}
V_0(R(0),\Omega(0),0)> 2a\frac{\sqrt{k_R}-\mu}{\sqrt{k_R}+\mu}k_R-\frac{1}{2k_\Delta}\delta^2.
\end{align}

Define $\theta_0\in[0,\pi]$ and $U_0$ as in \refeqn{R0:U:thetab}. For a constant $\epsilon\in(0,1)$, choose an angle $\theta_{b_0}\in(0,\theta_0)$ such that
\begin{equation}\label{eqn:theta_b0:4}
1-\cos(\theta_0-\theta_{b_0}) \leq
\parenth{2a\frac{\sqrt{k_R}-\mu}{\sqrt{k_R}+\mu}-\frac{1}{2k_\Delta k_R}\delta^2}\epsilon.
\end{equation}
Then, define the time-varying angle $\theta_b(t)$ as \refeqn{theta:b:varying} with a positive constant $\gamma$ satisfying
\begin{equation}\label{eqn:gamma:4}
\gamma < \frac{2}{\theta_{b_0}} \sqrt{2(1-\epsilon)\parenth{2a\frac{\sqrt{k_R}-\mu}{\sqrt{k_R}+\mu}k_R-\frac{1}{2k_\Delta}\delta^2}}.
\end{equation}
The shifted desired attitude is defined as \refeqn{Rd:shifted:varying}. While the shifted desired attitude in this section is constructed with different bounds on $\theta_{b_0}$ and $\gamma$, it is straightforward to show that it satisfies all of the properties summarized in Lemma \ref{lem:Rdtilde}.

We use $(\tilde R_d(t),\tilde \Omega_d(t))$ as a reference trajectory for the construction of the control system. The error variables $\tilde E_R$, $\tilde e_R$, and $\tilde e_\Omega$ are defined using the shifted reference trajectories. The control torque and the adaptive law are defined as
\begin{align}
\tilde \tau & =  - (\MI \Omega)\times \Omega \nonumber \\
&\quad+ \MI ( -k_R \tilde e_R - k_\Omega \tilde e_\Omega + \Omega \times \tilde\Omega_d + \dot{\tilde \Omega}_d)-\bar\Delta,\label{eqn:tracking:control:delta:tilde:varying}\\
\dot{\bar \Delta} & = k_\Delta \MI^{-1} (\tilde e_\Omega + \mu \tilde e_R),\label{eqn:bar:delta:dot:tilde:varying}
\end{align}
with $\bar\Delta(0)=0$.  According to Theorem \ref{theorem:main:robust}, for any initial condition satisfying
\begin{align}\label{eqn:ineq:var:V0:a:varying}
\frac{k_R}{4}\|R(0)-\tilde R_d(0)\|^2 &+ \frac{1}{2}\|\Omega(0)-\tilde \Omega_d(0)\|^2\nonumber\\
& \leq 2a\frac{\sqrt{k_R}-\mu}{\sqrt{k_R}+\mu}k_R-\frac{1}{2k_\Delta}\delta^2,
\end{align}
the tracking errors for the shifted desired trajectory, namely $\|R(t)-\tilde R_d(t)\|$, $\|\Omega(t)-\tilde\Omega_d(t)\|$, and $\|\Delta-\bar\Delta(t)\|$ asymptotically converge to zero. As the shifted reference trajectories are designed such that $\|R_d(t)-\tilde R_d(t)\|$ and $\|\Omega_d(t)-\tilde\Omega_d(t)\|$ exponentially converge to zero, these imply that all of the tracking errors $\|R(t)- R_d(t)\|$, $\|\Omega(t)-\Omega_d(t)\|$ from the original reference trajectories and the estimation error $\|\Delta-\bar\Delta(t)\|$ asymptotically converge to zero as $t\rightarrow \infty$. These are summarized as follows.

\begin{theorem}\label{thm:main_shifted:robust:varying}
Choose any positive numbers $k_R$, $k_\Omega$, $k_\Delta$, $a$, $\mu$, $\epsilon$, $\theta_{b_0}$, and $\gamma$ such that (\ref{ineq:a}), (\ref{ineq:epsilon}), \refeqn{kdelta}, \refeqn{theta_b0:4}, and \refeqn{gamma:4} are satisfied. The control input and the adaptive law are defined as follows.
\begin{subnumcases}
\; \tau =  - (\MI \Omega)\times \Omega + \MI ( - k_R  e_R -  k_\Omega e_\Omega + \Omega \times \Omega_d + \dot \Omega_d),\nonumber\\
\dot{\bar \Delta }= k_\Delta \MI^{-1}(e_\Omega+\mu  e_R),\nonumber\\
\text{when } V_0(R(0),\Omega(0),0)\leq 2a\tfrac{\sqrt{k_R}-\mu}{\sqrt{k_R}+\mu}k_R-\tfrac{1}{2k_\Delta}\delta^2,\label{eqn:tau_4a}\\
\; \tau =  - (\MI \Omega)\times \Omega + \MI ( - k_R \tilde e_R -  k_\Omega \tilde e_\Omega + \Omega \times \Omega_d + \dot \Omega_d),\nonumber\\
\dot{\bar \Delta }= k_\Delta \MI^{-1}(\tilde e_\Omega+\mu \tilde e_R),\nonumber\\
\text{when } V_0(R(0),\Omega(0),0)> 2a\tfrac{\sqrt{k_R}-\mu}{\sqrt{k_R}+\mu}k_R-\tfrac{1}{2k_\Delta}\delta^2,
\label{eqn:tau_4b}
\end{subnumcases}
with $\bar\Delta(0)=0$ for both cases. Then, the zero equilibrium of the tracking errors $(E_R,e_\Omega,e_\Delta)=(0,0,0)$ is asymptotically stable.
\end{theorem}

The corresponding region of attraction is characterized in the following corollary.

\begin{corollary}
For the control system presented in Theorem \ref{thm:main_shifted:robust:varying}, the region of attraction guaranteeing asymptotic convergence encloses the following set,
\begin{align}\label{eqn:ROA:4}
&\mathcal{R}=\{(R(0),\Omega(0))\in\SO\times\Re^3\,|\, \|e_\Omega(0)\| <\nonumber \\
&\max\{
\sqrt{B-k_R(1-\cos\theta_0)},\,\sqrt{2(1-\epsilon)B}-\frac{\gamma}{2}\theta_{b_0}\} \},
\end{align}
where $\theta_0$ is constructed from $R(0)$ by \refeqn{R0:U:thetab}, and the positive constant $B$ is defined as
\begin{equation}
B = 2a\frac{\sqrt{k_R}-\mu}{\sqrt{k_R}+\mu}k_R -\frac{1}{2k_\Delta}\delta^2.
\end{equation}

Furthermore, $\mathcal{R}\subset\SO\times\Re^3$ covers $\SO$ completely, and when projected onto $\Re^3$, it is enlarged into $\Re^3$ as $k_R$ is increased in a semi-global sense.
\end{corollary}

\begin{proof}
Define three subsets of $\SO\times\Re^3\times\Re^3$ for the initial condition as
\begin{align*}
\mathcal{R}_1&=\bigg\{k_R(1-\cos\theta_0)+\frac{1}{2}\|e_\Omega(0)\|^2\leq B \bigg\},\\
\mathcal{R}_2&=\bigg\{k_R(1-\cos\theta_0)+\frac{1}{2}\|e_\Omega(0)\|^2 > B \bigg\},\\
\mathcal{R}_3&=\bigg\{k_R(1-\cos(\theta_0-\theta_{b_0}))+\frac{1}{2}\|\tilde e_\Omega(0)\|^2\leq B \bigg\}.
\end{align*}
The sets $\mathcal{R}_1$ and $\mathcal{R}_2$ represent the initial conditions corresponding to the two cases of the control inputs, namely \refeqn{tau_4a} and \refeqn{tau_4b}, respectively. The set $\mathcal{R}_3$ represents the set of initial conditions \refeqn{ineq:var:V0:a:varying}, guaranteeing the asymptotic convergence for the second case of the control input. Therefore, the combined region of attraction $\bar{\mathcal{R}}$ is given by $\bar{\mathcal{R}}\equiv\mathcal{R}_1\cup(\mathcal{R}_2\cap\mathcal{R}_3)$. Since $\mathcal{R}_1\cup\mathcal{R}_2=\SO\times\Re^3$, this reduces to $\bar{\mathcal{R}}=(\mathcal{R}_1\cup\mathcal{R}_2)\cap
(\mathcal{R}_1\cup\mathcal{R}_3)=\mathcal{R}_1\cup\mathcal{R}_3$.

Next, we show that $\mathcal{R}\subset \bar{\mathcal{R}}$. For any $(R(0),\Omega(0))\in\mathcal{R}$,
\[
\|e_\Omega(0)\|\leq \sqrt{B-k_R(1-\cos\theta_0)},
\]
or
\[
\|e_\Omega(0)\|\leq\sqrt{2(1-\epsilon)B}-\frac{\gamma}{2}\theta_{b_0},
\]
where the right-hand side is positive due to \refeqn{gamma:4}. For the former case, $(R(0),\Omega(0))\in\mathcal{R}_1 \subset \bar{\mathcal{R}}$. For the latter case, from \refeqn{tilde:eW:bound}, $
\|\tilde e_\Omega(0)\|\leq \sqrt{2(1-\epsilon)B}$. Hence,
\begin{align*}
k_R & (1-\cos(\theta_0-\theta_{b_0}))+\frac{1}{2}\|\tilde e_\Omega(0)\|^2 \\
&\leq k_R (1-\cos(\theta_0-\theta_{b_0})) + (1-\epsilon)B \leq B,
\end{align*}
from \refeqn{theta_b0:4}. Therefore, $(R(0),\Omega(0))\in\mathcal{R}_3$. As such, any initial condition in $\mathcal{R}$ also belongs to the estimated region of attraction $\bar{\mathcal{R}}$, and the resulting controlled trajectory asymptotically converges to the desired one.

Also, for any $\theta_0\in[0,\pi]$, the set $\mathcal{R}$ is non-empty due to \refeqn{gamma:4}. At \refeqn{ROA:4}, the upper bound on $\|e_\Omega(0)\|$ tends to be infinite as $k_R\rightarrow\infty$.
\end{proof}

As in the previous section, the region attraction covers the entire $\SO$, and when projected to $\Re^3$ for the angular velocity, it is enlarged in a semi-global sense as $k_R$ increases. But, such global attractivity on $\SO$ is achieved with the control input and the adaptive law that are formulated as continuous functions of $t$. As discussed in Section \ref{sec:SGT}, overcoming the topological restriction on the attitude control with a non-switching control input has been unprecedented. The adaptive law presented in this section additionally allows a constant disturbance to be present in the dynamics  at the expense of sacrificing the exponential convergence.

\section{Numerical Examples}

\begin{figure}
\centerline{
	\subfigure[Tracking errors]{\includegraphics[width=0.8\columnwidth]{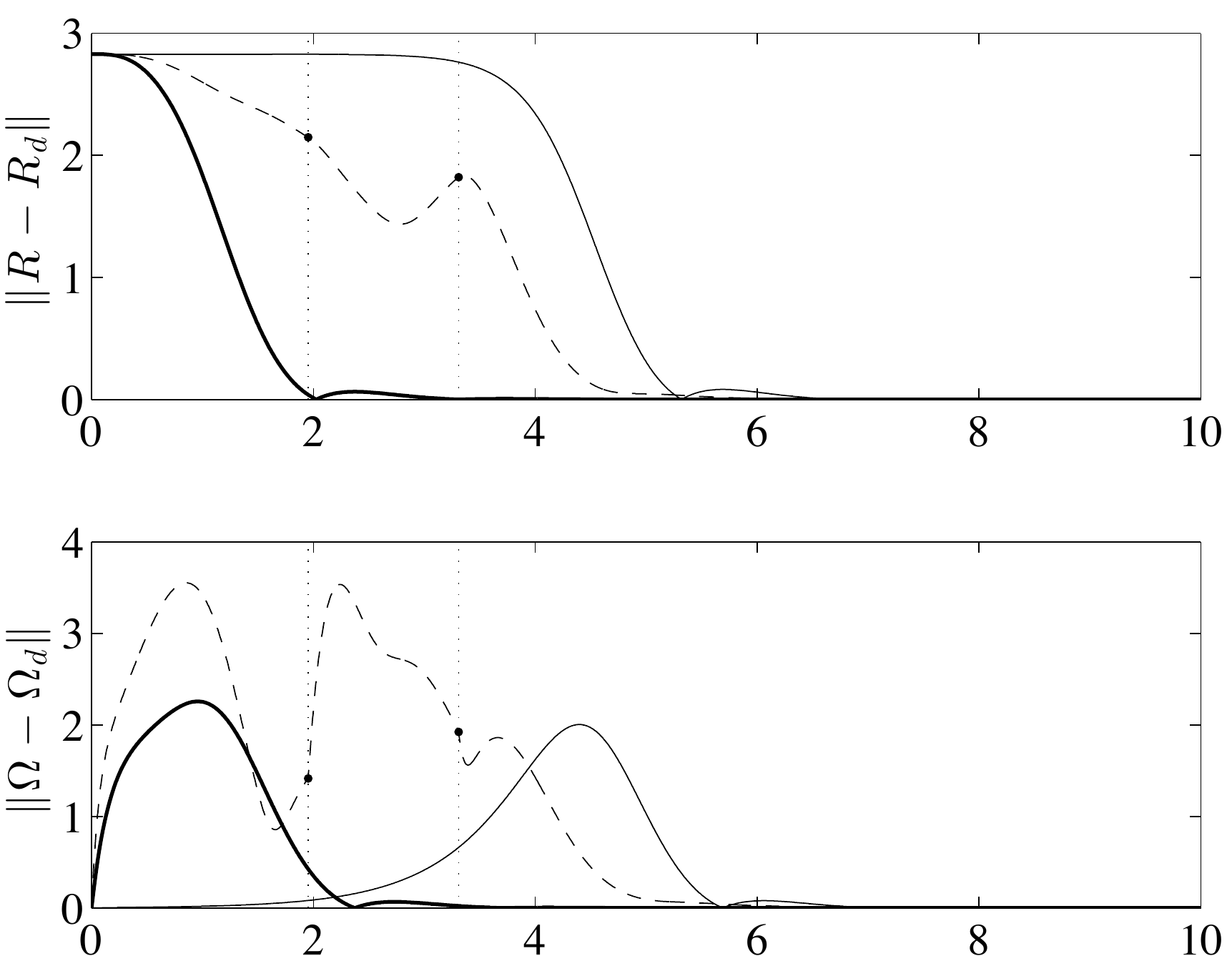}}
	}
\centerline{
	\subfigure[Control inputs]{\includegraphics[width=0.8\columnwidth]{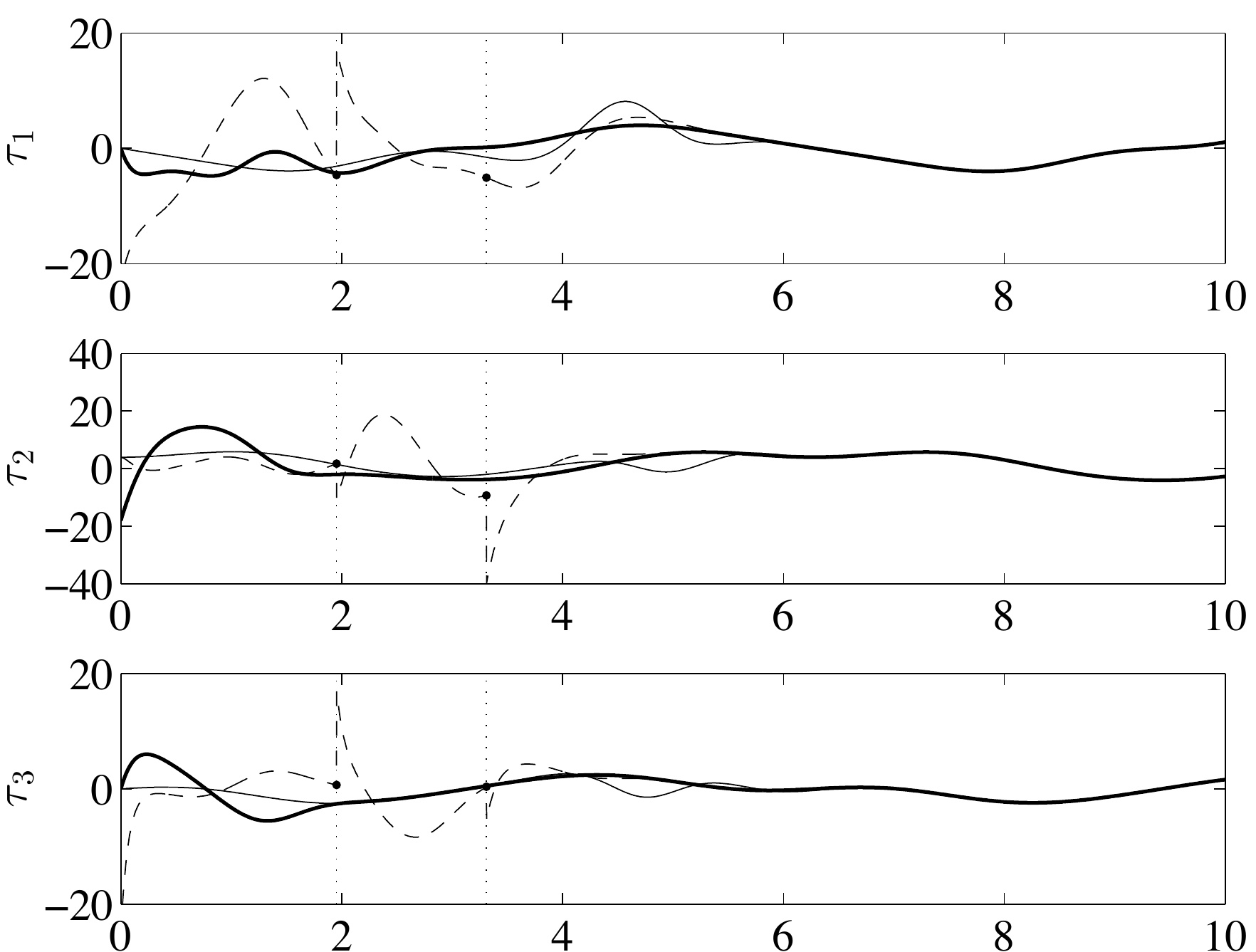}}
	}
\caption{Simulation results for two attitude tracking controls presented in Section \ref{sec:ATC} and a hybrid attitude control~\cite{LeeITAC15} (AGTS:solid, GTS:bold, HYB:dotted; the switching time for HYB is denoted by vertical dotted lines).}\label{fig:Sec3}
\end{figure}

\subsection{Attitude Tracking Controls}

We first show the numerical results for the attitude tracking controls presented in Section \ref{sec:ATC}. Throughout this section, two control strategies presented in Section \ref{sec:ATC} are denoted by AGTS (almost global tracking strategy), and GTS (global tracking strategy), respectively. They are also compared with the hybrid attitude control presented in~\cite{LeeITAC15} that guarantees global exponential stability with discontinuities of the control input with respect to $t$, and it is denoted by HYB.

Assume that the moment of inertia matrix of the system is
\[
\MI = \operatorname{diag}[3, 2, 1]\,\mathrm{kgm^2}.
\]
Consider the following reference trajectory:
\begin{align*}
&R_d(t) \\
&{\footnotesize = \begin{bmatrix}
\cos t & -\cos t \sin t & \sin^2 t \\
\cos t \sin t & \cos^3 t -\sin^2 t & -\cos t \sin t - \cos^2t \sin t \\
\sin^2t & \cos t \sin t +\cos^2 t \sin t & \cos^2t - \cos t \sin^2t
\end{bmatrix}}
\end{align*}
which implies
\begin{align*}
\Omega_d(t) &= \begin{bmatrix}
1+\cos t \\ \sin t - \sin t \cos t \\ \cos t + \sin^2 t
\end{bmatrix}, \\
 \dot \Omega_d (t) &= \begin{bmatrix}
-\sin t\\ \cos t -\cos^2t + \sin^2 t\\
-\sin t + 2\sin t \cos t
\end{bmatrix}.
\end{align*}
Notice that
\[
R_d(0) =  \operatorname{diag}[1, 1, 1], \quad \Omega_d(0) = (2,0, 1).
\]
The initial state of the system is given by
\[
R(0) = R_d(0)\exp (\theta_0\hat e_2), \quad \Omega(0) = (2,0,1),
\]
where $\theta_0=0.999\pi$ and $e_2 = (0,1,0)$.

The controller parameters are chosen as
\begin{gather*}
k_R = 9,\quad k_\Omega=4.2,\quad a=\epsilon,\quad
\mu = \frac{4(1-a)k_R k_\Omega}{4(1-a) k_R + k_\Omega^2}\epsilon,\\
\theta_{b_0} = \min\{\theta_0\epsilon,\theta_0-\cos^{-1}(1-2a\epsilon)\},\\
\gamma=\frac{4}{\theta_{b_0}}\sqrt{a k_R (1-\epsilon)}\epsilon,
\end{gather*}
with a scaling parameter $\epsilon=0.9<1$ selected to satisfy the inequality constraints. These expressions allow that all of the controller parameters can be determined by tuning only the proportional gain $k_R$ and the derivative gain $k_\Omega$.


The corresponding simulation results are plotted in Figure \ref{fig:Sec3}. For the given initial conditions and the controller gains, $18=V_0(0) > 2ak_R=16.2$. Therefore, the initial condition does not belong to the region of attraction of AGTS estimated conservatively by (\ref{V:ini}) in spite of which the tracking errors for AGTS still converge to zero asymptotically.  The convergence rate is, however, quite low, and there is no noticeable change of the attitude tracking error for the first three seconds. The slow initial convergence is common for controllers with almost global asymptotic stability, especially when the initial state is close to the stable manifold of the undesired equilibrium point \cite{LeeLeoPICDC11}.

Next, for HYB, the convergence rate for the attitude tracking error is improved. But the tracking error for the angular velocity is increased over the first four seconds, and there are two abrupt  changes in the control input at $t=1.95$ and $t=3.31$. Compared with AGTS, the convergence rate is substantially improved at the expense of discontinuities in the control input.

Finally, the proposed GTS exhibits the fastest convergence rate for both of the attitude tracking errors and the angular velocity tracking errors. This is most desirable, as excellent tracking performances are achieved without discontinuities in control input for the large initial attitude error.

\begin{figure}
\centerline{
	\subfigure[Tracking errors and estimation error]{\includegraphics[width=0.8\columnwidth]{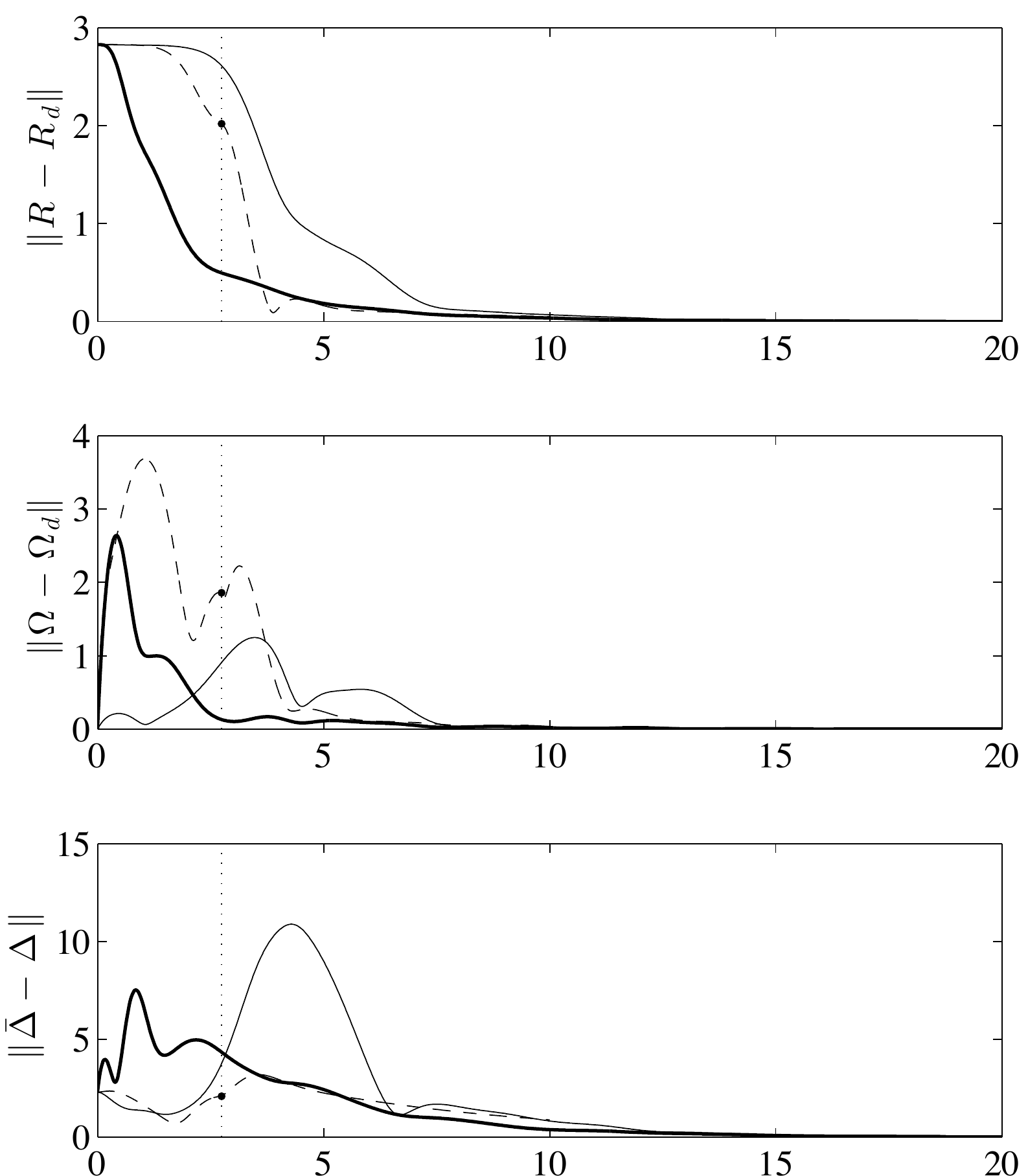}}
	}
\centerline{
	\subfigure[Control inputs]{\includegraphics[width=0.8\columnwidth]{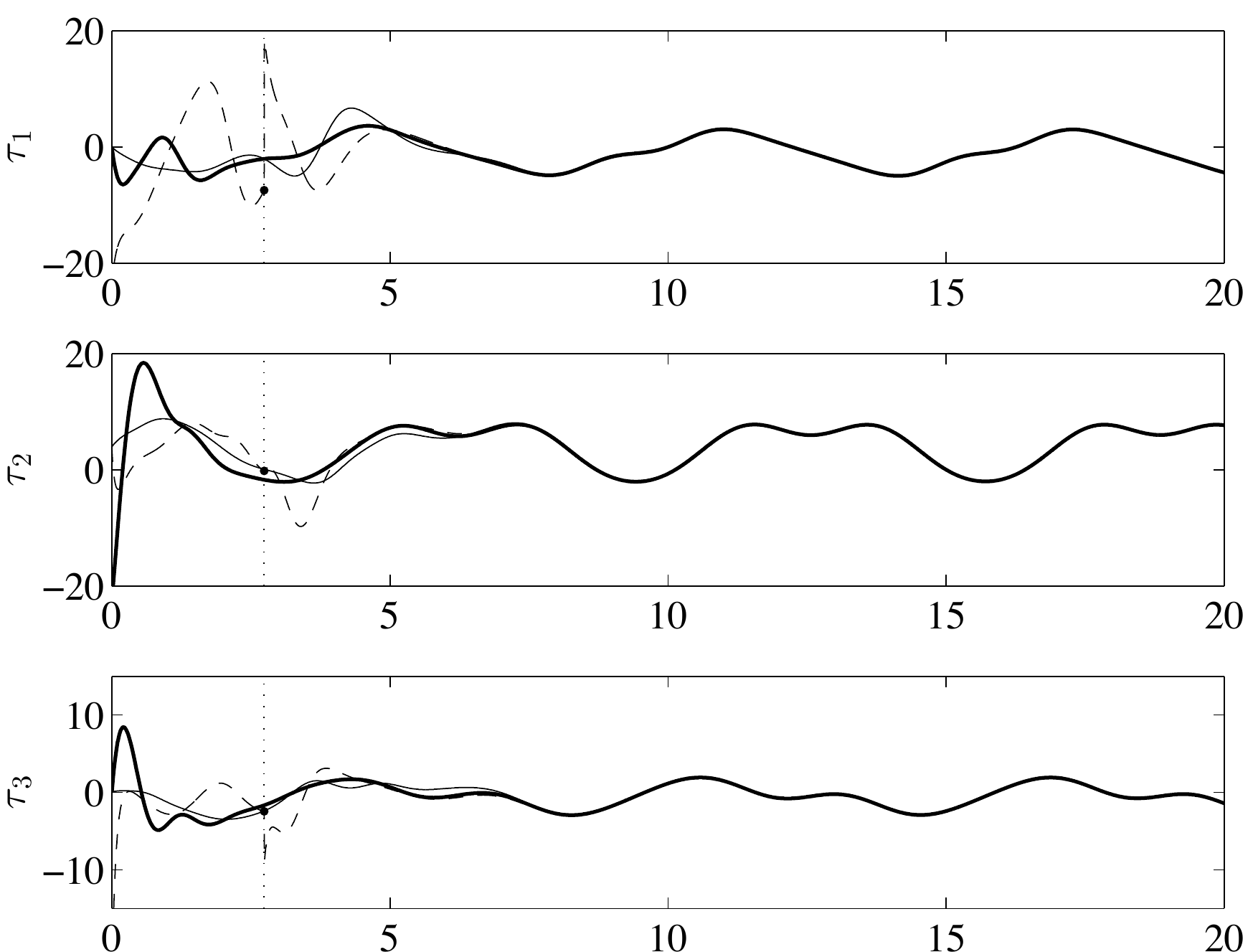}}
	}
\caption{Simulation results for two adaptive attitude tracking controls presented in Section \ref{sec:AATC} and an adaptive hybrid control~\cite{LeeITAC15} (aAGTS:solid, aGTS:bold, aHYB:dotted; the switching time for aHYB is denoted by vertical dotted lines).}\label{fig:Sec4}
\end{figure}

\subsection{Adaptive Attitude Controls}\label{sec:NEb}

Next, we present the simulation results for the adaptive attitude control strategies presented in Section \ref{sec:AATC}. Throughout this section, two control strategies of Section \ref{sec:AATC} are denoted by aAGTS (adaptive almost global tracking strategy), and aGTS (adaptive global tracking strategy), respectively. They are compared with the extension of the hybrid attitude control with an adaptive term in~\cite{LeeITAC15}, and it is denoted by aHYB.

The fixed disturbance is chosen as
\[
\Delta = (1,-2,0.5),
\]
with the estimated bound of $\delta = 3$. The values of the controller parameters $k_R$, $k_\Omega$, $a$, and $\mu$ are identical to those used in the previous subsection. The other parameters are chosen as
\begin{gather*}
\theta_{b_0} = \min\{\theta_0\epsilon,\theta_0-\cos^{-1}(1-B\epsilon/k_R)\},\\
\gamma=\frac{2}{\theta_{b_0}}\sqrt{2(1-\epsilon)B}\epsilon,\quad
k_\Delta=25,
\end{gather*}
with a scaling parameter $\epsilon=0.9<1$ such that the inequality constraints are satisfied. One can easily verify that the chosen value of $k_\Delta$ satisfies \refeqn{kdelta}. In short, all of the controller parameters can be selected by tuning $k_R$, $k_\Omega$, and $k_\Delta$ with consideration of \refeqn{kdelta} for $k_\Delta$.

Simulation results are plotted in Figure \ref{fig:Sec4}. All the three controllers successfully estimate the unknown disturbance, as the estimation error $\bar\Delta-\Delta$ asymptotically converges to zero, and the effects of the disturbance are mitigated. The overall performance characteristics for each method are similar to those in the previous subsection.

For aAGTS, the given initial condition does not belong to the region of attraction estimated by \refeqn{ineq:var:V0}, as $18=V_0(0)> B=10.31$. However, both the tracking errors and the estimation error for aAGTS asymptotically converge to zero, although the initial convergence rate is low.

With aHYB, the attitude tracking performance is substantially improved, but the initial angular velocity tracking error is increased. There exists a discontinuity in the control input at $t=2.74$, where the magnitude of  control moment change at the jump exceeds $26.8\,\mathrm{Nm}$.

The proposed aGTS exhibits the best tracking performance for the attitude and the angular velocity, and it does not cause any discontinuity of the control input.

\section{Experimental Results}

\begin{figure}
\centerline{
	\subfigure[Initial attitude]{\includegraphics[width=0.48\columnwidth]{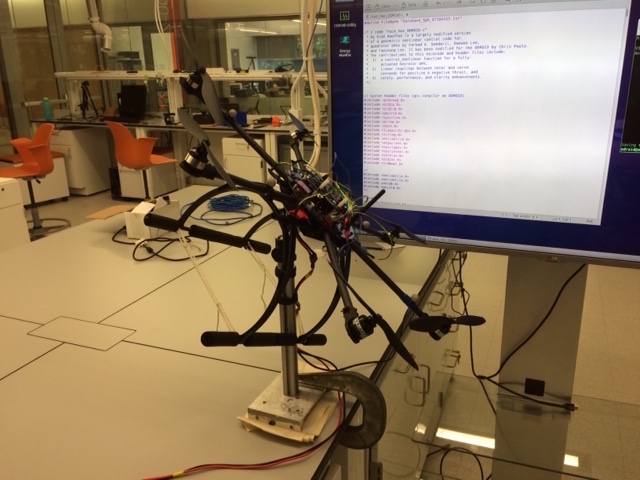}}
	\hfill
	\subfigure[Stabilizing the inverted equilibrium]{\includegraphics[width=0.48\columnwidth]{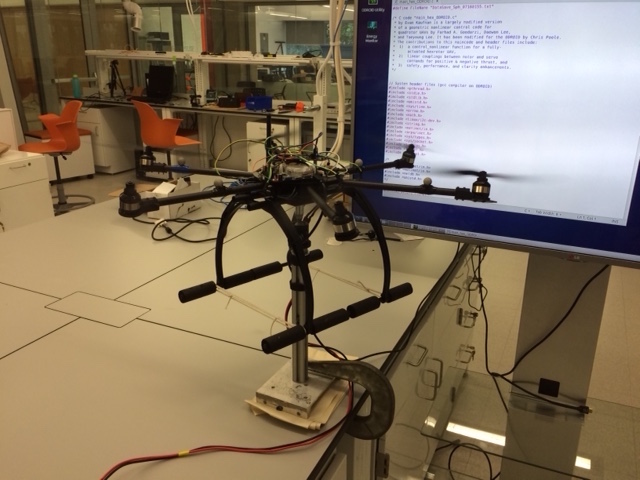}}
}
\caption{Hexrotor attitude control experiment}
\end{figure}

\begin{figure}
\centerline{
	\subfigure[Tracking errors and estimated disturbance]{\includegraphics[width=0.8\columnwidth]{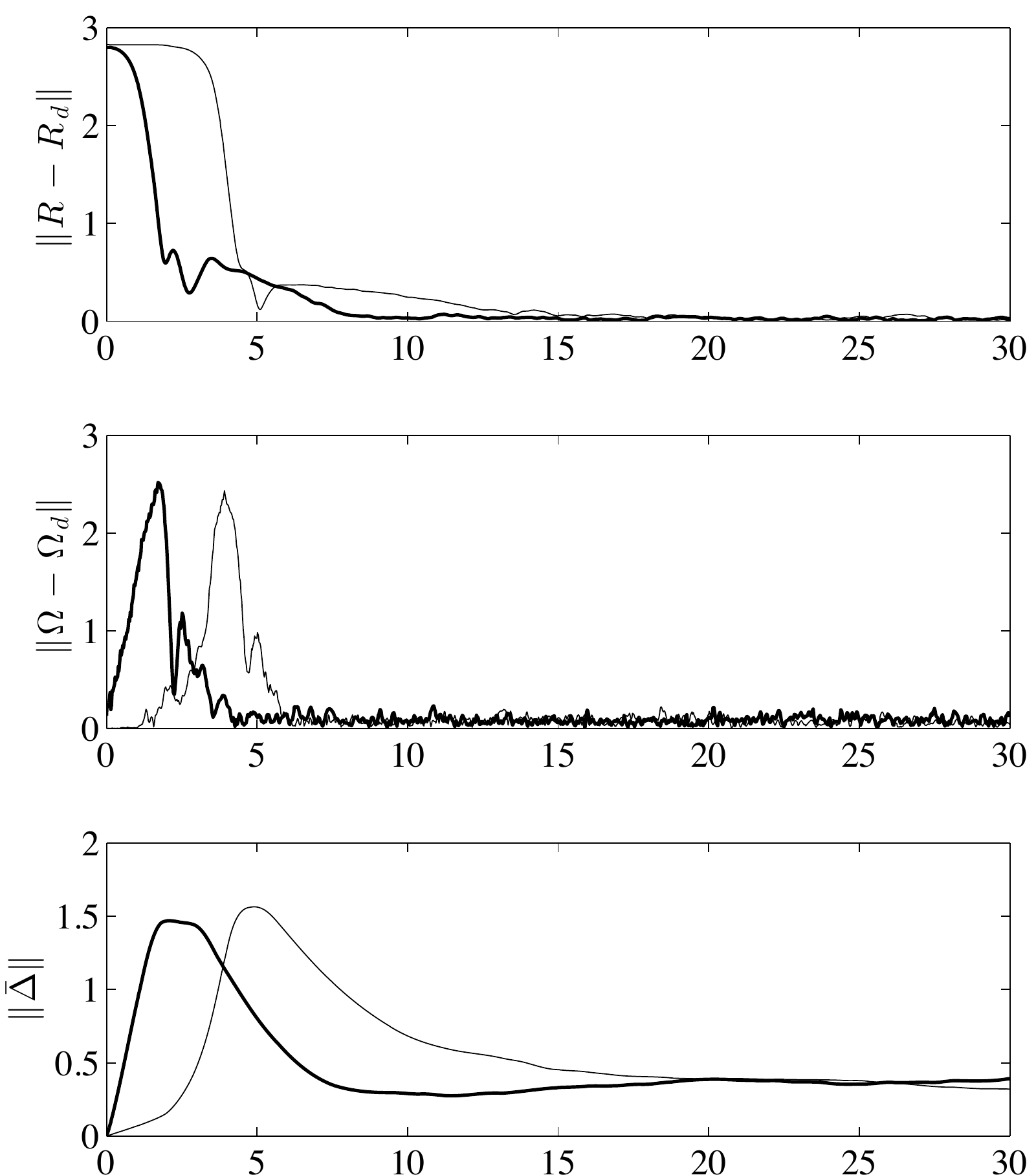}}
	}
\centerline{
	\subfigure[Control inputs]{\includegraphics[width=0.8\columnwidth]{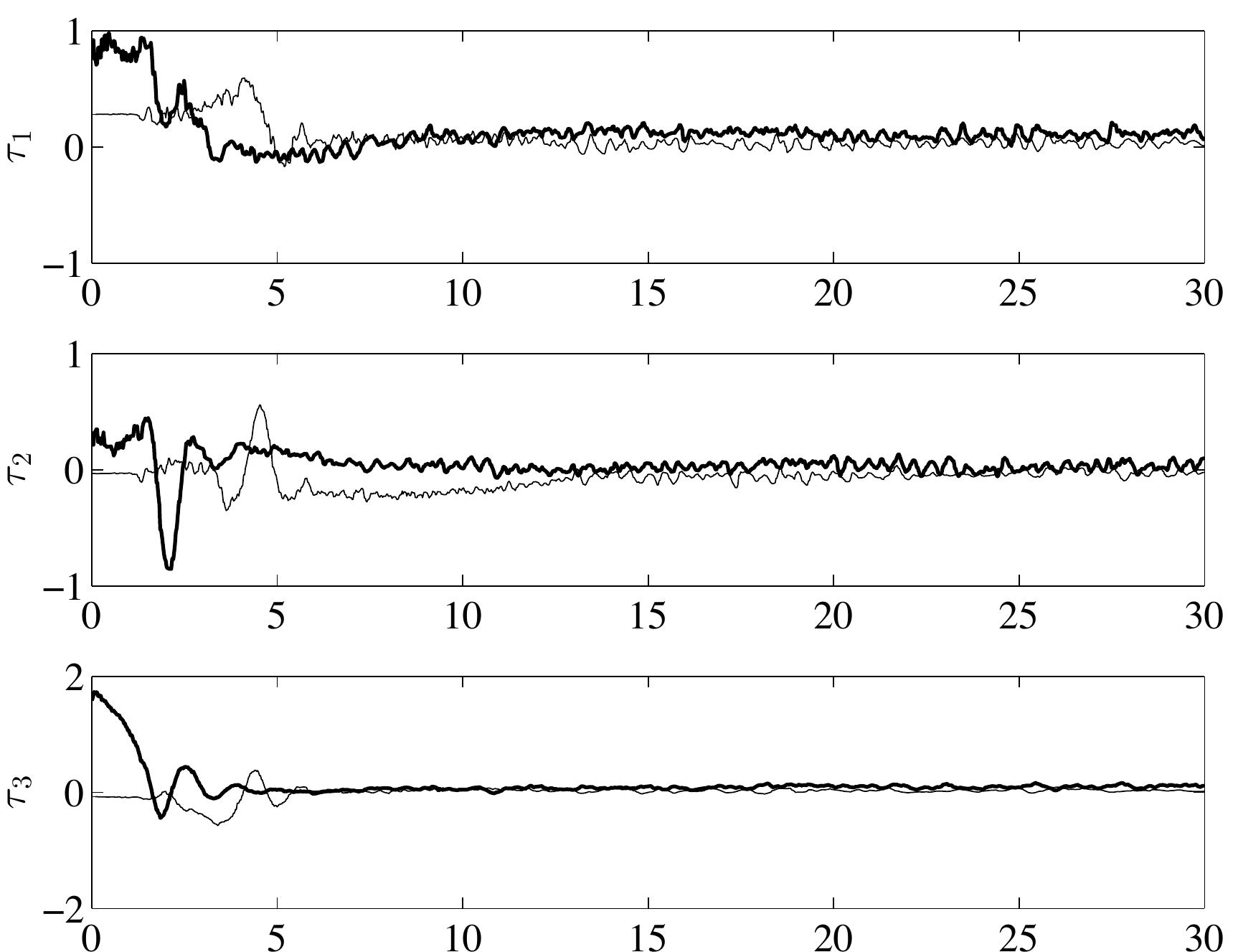}}
	}
\caption{Experimental results for two adaptive attitude tracking controls presented in Section \ref{sec:AATC} (aAGTS:solid, aGTS:bold)}\label{fig:Sec6}
\end{figure}

The two  adaptive attitude tracking control strategies presented in Section \ref{sec:AATC} have been implemented on the attitude dynamics of a hexrotor unmanned aerial vehicle, to illustrate the efficacy of the proposed approaches through hardware experiments.

\subsection{Hexrotor Development}

The hardware configuration of the hexrotor is as follows. Six brushless DC motors (Robbe Roxxy) are used along with electric speed controllers (Mikrokopter BL-Ctrl 2.0). An inertial measurement unit (VectorNav VN-100) provides the angular velocity and the attitude of the hexrotor. A linux-based computing module (Odroid XU-3) handles onboard data processing, sensor fusion, control input computation, and communication with a host computer (Macbook Pro). A custom made, printed circuit board supplies power to each part from a battery after switching the voltage level appropriately.

A flight software program is developed by utilizing multithread programming in gcc such that the tasks of communication, sensor fusion, and control are performed in a parallel fashion. In particular, the control input is computed at the rate of $120\,\mathrm{Hz}$ approximately.

The hexrotor is attached to a spherical joint that provides unlimited rotation in the yaw direction, and $\pm 45^\circ$ rotations along the pitch and the roll.  As the center of gravity is above the center of the spherical joint, it resembles an inverted rigid body pendulum~\cite{ChaLeeJNS09}.

\subsection{Experimental Results}

Two adaptive attitude tracking control strategies that provide smooth control inputs, namely aAGTS and aGTS are implemented. The desired attitude corresponds to the inverted equilibrium, where the center of gravity of the hexrotor is directly above the spherical joint, and the first body-fixed axis of the inertial measurement unit points towards the magnetic north, i.e., $R_d(t)=I_{3\times 3}$ for all $t\geq 0$. Note that the desired attitude is unstable due to the gravity.

The initial condition is chosen such that the pitch angles is decreased until the limit of the spherical joint, and the first body-fixed axis points towards the magnetic south. The resulting initial attitude error is close to $180^\circ$, i.e., $\|E_R(0)\|\simeq 2\sqrt{2}$. The initial angular velocity is chosen as zero. The controller parameters are selected as $k_R=1.45$, $k_\Omega=0.4$, $k_\Delta=0.2$, and $\delta=1$. Other parameters are identical to those presented in Section \ref{sec:NEb}.

The corresponding experimental results are illustrated in Figure \ref{fig:Sec6}. The overall behaviors of adaptive controllers are similar to the numerical simulation results. The unstable desired attitude is asymptotically stabilized by both adaptive attitude controllers.

For aAGTS, the initial convergence rate, particularly for the attitude tracking error, is quite slow. For example, the attitude tracking error remains close to its initial value for the first few seconds. However, aGTS exhibits a satisfactory convergence rate from the beginning, and it shows most desirable results.

\section{Conclusions}

We have proposed global tracking strategies for the attitude dynamics of a rigid body. The topological restriction on the special orthogonal group is circumvented by introducing a shifted reference trajectory with a conjugacy class. As a result, global attractivity is achieved without causing discontinuities in  control input with respect to time. These are constructed on the special orthogonal group to avoid singularities and ambiguities of other attitude representations. The desirable properties of the proposed methods are demonstrated by numerical examples and experimental results.

The proposed approaches are fundamentally distinctive from the current efforts to achieve global attractivity via modifying attitude configuration error functions along with the hybrid system framework. This paper shows that the desired trajectory can be adjusted instead, and global attractivity can be achieved without introducing undesired abrupt changes in control input.

For future work, the idea of shifting reference trajectories can be applied to feedback control on other compact manifolds and Lie groups. Also, the results presented in this paper are readily generalized to various other attitude control problems such as velocity-free attitude controls or deterministic attitude observers.

\section*{Acknowledgement}
This work was supported in part by NSF under the grants CMMI-1243000, CMMI-1335008, and CNS-1337722, and by KAIST under grant G04170001. It was also supported in part by DGIST Research and Development Program (CPS Global Center) funded by the Ministry of Science, ICT \& Future Planning, Global Research Laboratory Program (2013K1A1A2A02078326) through NRF, and Institute for Information \& Communications Technology Promotion (IITP) grant funded by the Korean government (MSIP) (No. 2014-0-00065, Resilient Cyber-Physical Systems Research).

\bibliography{BibMaster17,tylee}
\bibliographystyle{IEEEtran}

\end{document}